\documentclass[a4paper,12pt,reqno]{amsart}
\usepackage{a4wide}
\usepackage{amsmath}
\usepackage{amssymb}
\usepackage{amsthm}
\usepackage{latexsym}
\usepackage{graphicx}
\usepackage[english]{babel}
              
\newtheorem{prop}[subsection]{Proposition}
\newtheorem{conj}[subsection]{Conjecture}
\newtheorem{teor}[subsection]{Theorem}
\newtheorem{lema}[subsection]{Lemma}
\newtheorem{cor} [subsection]{Corollary}
\theoremstyle{definition}

\theoremstyle{remark}
\newtheorem{obs} [subsection]{Remark}
\newtheorem{exm} [subsection]{Example}

\newcommand{\me}{\mathfrak m}

\def\qdepth{\operatorname{hdepth}}
\def\hdepth{\operatorname{hdepth}}
\def\depth{\operatorname{depth}}

\def\deg{\operatorname{deg}}

\numberwithin{equation}{section}

\begin{document}

\title[Comparing Hilbert depth of $I$ with Hilbert depth of $S/I$]{Comparing Hilbert depth of $I$ with Hilbert depth of $S/I$}
\author[Andreea I.\ Bordianu, Mircea Cimpoea\c s 
       ]
  {Andreea I.\ Bordianu$^1$ and Mircea Cimpoea\c s$^2$
	}
\date{}



\footnotetext[1]{ \emph{Andreea I.\ Bordianu}, University Politehnica of Bucharest, Faculty of
Applied Sciences, 
Bucharest, 060042, E-mail: andreea.bordianu@stud.fsa.upb.ro}
\footnotetext[2]{ \emph{Mircea Cimpoea\c s}, University Politehnica of Bucharest, Faculty of
Applied Sciences, 
Bucharest, 060042, Romania and Simion Stoilow Institute of Mathematics, Research unit 5, P.O.Box 1-764,
Bucharest 014700, Romania, E-mail: mircea.cimpoeas@upb.ro,\;mircea.cimpoeas@imar.ro}

\begin{abstract}
Let $I$ be a monomial ideal of $S=K[x_1,\ldots,x_n]$.
We show that the following are equivalent: (i) $I$ is principal, (ii) $\qdepth(I)=n$, (iii) $\qdepth(S/I)=n-1$.

Assuming that $I$ is squarefree, we prove that if $\qdepth(S/I)\leq 3$ or $n\leq 5$ then $\qdepth(I)\geq \qdepth(S/I)+1$.
Also, we prove that if $\qdepth(S/I)\leq 5$ or $n\leq 7$ then $\qdepth(I)\geq \qdepth(S/I)$.
\end{abstract}

\maketitle
\section{Introduction}

Let $K$ be a field and $S=K[x_1,\ldots,x_n]$ be the polynomial ring over $K$.
Given a finitely generated graded $S$-module $M$, the \emph{Hilbert depth} of $M$, denoted by $\hdepth(M)$, is the 
maximal depth of a finitely generated graded $S$-module $N$ with the same Hilbert series as $M$; see \cite{bruns} and 
\cite{uli} for further details. 

One would expect that it is easy to compute the Hilbert depth of a module, once its Hilbert function is known. 
But it turns out that even for the powers of the maximal ideal, the computation of the Hilbert depth leads to difficult
numerical computations; see \cite{maxim}. Another argument for studying this invariant is the fact that the Hilbert depth
of a finitely generated $\mathbb Z^n$-graded $S$-module $M$ is an upper bound for the Stanley depth of $M$; for further
details on this topic we refer the reader to \cite{herzog}.
 
Let $0\subset I\subsetneq J\subset S$ be two squarefree monomial ideals.
In \cite{lucrare2}, the authors presented a new method for computing the Hilbert depth of $J/I$, as follows:
We consider the poset
$$P_{J/I}=\{A\subset [n]\;:\;x_A=\prod_{j\in A}x_j \in J\setminus I\} \subset 2^{[n]}.$$ 
We let $$\alpha_j(J/I)=|\{A\in P_{J/I}\;:\;|A|=j\}|,\text{ for }0\leq j\leq n.$$
For all $0\leq q\leq n$ and $0\leq k\leq q$, we consider the integers
\begin{equation}\label{betak}
  \beta_k^q(J/I):=\sum_{j=0}^k (-1)^{k-j} \binom{q-j}{k-j} \alpha_j(J/I).
\end{equation}
Note that, using an inverse formula, from \eqref{betak} we deduce that
\begin{equation}\label{alfak}
  \alpha_k(J/I)=\sum_{j=0}^k \binom{q-j}{k-j} \beta^q_j(J/I),\text{ for all }0\leq k\leq q\leq n.
\end{equation}
With the above notations, we have the following result:

\begin{teor}(\cite[Theorem 2.4]{lucrare2})\label{d1}
The Hilbert depth of $J/I$ is:
$$\qdepth(J/I)=\max\{q\;:\;\beta_k^q(J/I) \geq 0\text{ for all }0\leq k\leq q\}.$$
\end{teor}

From the proof of the above theorem, we note the fact:

\begin{cor}\label{max}
If $\hdepth(J/I)=q$, then $\beta_k^{q'}(J/I) \geq 0$ for all $0\leq k\leq q'\leq q$.
\end{cor}

If $I\subset J\subset S$ are two monomial ideals, then we consider their polarizations $I^p\subset J^p\subset R$,
where $R$ is a new ring of polynomials obtained from $S$ by adding $N$ new variables. The following proposition
shows that we can reduce the study of the Hilbert depth of a quotient of monomial ideals to the
squarefree case:

\begin{prop}(\cite[Proposition 2.6]{lucrare2})\label{d2}
The Hilbert depth of $J/I$ is
$$\qdepth(J/I)=\qdepth(J^p/I^p)-N.$$
\end{prop}

Note that, if $I\subset J\subset S$ are two monomial ideals, then 
\begin{equation}\label{t11}
\depth(J/I)\leq \qdepth(J/I) \leq \dim(J/I).
\end{equation}

The aim of our paper is to continue the study of the Hilbert depth of monomial ideals, using the above method.
In Theorem \ref{teo1} we prove that the following are equivalent for a monomial ideal $I\subset S$: 
$$ (i)\; I\text{ is principal}, \;(ii)\; \qdepth(I)=n\text{ and }(iii)\;\qdepth(S/I)=n-1.$$
Note that $\qdepth(I)=\qdepth(S/I)+1$ if $I$ is principal. More generally, if $S/I$ is Cohen-Macaulay,
then, according to \cite[Theorem 2.8]{lucrare3}, we have $$\qdepth(I)\geq \qdepth(S/I)+1.$$
It is natural to ask if this equality remains true in the non Cohen-Macaulay case. 
In general, the answer is no. However, we will see that in some special cases, we can still compare
$\qdepth(S/I)$ with $\qdepth(I)$. In order to do that, we rely heavily on the famous Kruskal-Katona Theorem; 
see Theorem \ref{krusk}.

In Theorem \ref{teo2} we prove that if $I\subset S$ is squarefree with $\qdepth(S/I)\leq 3$, then 
$$\qdepth(I)\geq \qdepth(S/I)+1.$$
We also show that this inequality holds if $I\subset K[x_1,\ldots,x_5]$ is squarefree; see Corollary \ref{cor2}. 
The above results are sharp, in the sense that there we provide an example of a squarefree ideal 
$I\subset K[x_1,\ldots,x_6]$ with $\qdepth(S/I)=\qdepth(I)=4$; see Example \ref{patru}.

In Theorem \ref{t3} we prove that if $I\subset S$ is squarefree with $\qdepth(S/I)=4$, then $\qdepth(I)\geq 4$.
Also, in Corollary \ref{cor3} we show that if $I\subset K[x_1,\ldots,x_6]$ is squarefree, then
$\qdepth(I)\geq \qdepth(S/I)$.

In Example \ref{minus}, we provide a squarefree monomial ideal $I\subset K[x_1,\ldots,x_{13}]$ with 
$\qdepth(I)=7<\qdepth(S/I)=8$. In Example \ref{minus2}, we provide a squarefree monomial ideal $I\subset K[x_1,\ldots,x_{14}]$ with
$\qdepth(I)=6<\qdepth(S/I)=7$. In Example \ref{minus3}, we provide a squarefree monomial ideal $I\subset K[x_1,\ldots,x_{10}]$ with
$\qdepth(I)=6<\qdepth(S/I)=7$.

This yields us to conjecture that $\qdepth(I)\geq \qdepth(S/I)$ if $\qdepth(S/I)\leq 6$ or $n\leq 9$; see Conjecture \ref{conj1} and Conjecture \ref{conj2}.
In the last section we tackle the case $\qdepth(S/I)=5$ and we show that if $\qdepth(S/I)=5$ or $n=7$, then $\qdepth(I)\geq \qdepth(S/I)$;
see Theorem \ref{main} and Corollary \ref{cmain}.

\section{Preliminary results}

First, we prove the following combinatorial identity:

\begin{lema}\label{magic}
For all $0\leq k\leq q\leq n$ we have that
$$\sum_{j=0}^k (-1)^{k-j}\binom{q-j}{k-j}\binom{n}{j} = \binom{n-q+k-1}{k}.$$
\end{lema}

\begin{proof}
By the Chu-Vandermonde summation, we have
$$\sum_{j=0}^k (-1)^{k-j}\binom{q-j}{k-j}\binom{n}{j} = \sum_{j=0}^k \binom{-q+k-1}{k-j}\binom{n}{j} =\binom{n-q+k-1}{k},$$
as required.
\end{proof}


\begin{lema}\label{liema}
Let $I\subset S$ be a squarefree monomial ideal. The following are equivalent:
\begin{enumerate}
\item[(1)] $I=(u)$, where $u\in S$ is a squarefree monomial of degree $m$ with $1\leq m\leq n$.
\item[(2)] $\alpha_k(I)=\binom{n-m}{k-m}$ for all $0\leq k\leq n$.
\item[(3)] $\alpha_k(S/I)=\binom{n}{k}-\binom{n-m}{k-m}$ for all $0\leq k\leq n$.
\item[(4)] $\beta_k^n(I)=\delta_{km}=\begin{cases} 1,& k=m \\0,& k\neq m \end{cases}$.
\item[(5)] $\beta_k^{n-1}(S/I)=\begin{cases} 1,& 0\leq k\leq m-1 \\ 0,& m\leq k\leq n \end{cases}$.
\end{enumerate}
\end{lema}

\begin{proof}
$(1)\Rightarrow (2)$. We can assume that $u=x_{n-m+1}\cdots x_n$. It follows that a squarefree monomial $v$ of degree $k$ 
belongs to $I$ if and only if $v=u\cdot w$, where $w\in K[x_1,\ldots,x_{n-m}]$ is squarefree with $\deg(w)=k-m$. Thus, we are done.

$(2)\Rightarrow (1)$. Since $\alpha_m(I)=1$, it follows that there exists a squarefree monomial $u\in I$, of degree $m$.
It follows that $L:=(u)\subset I$. From $(1)\Rightarrow (2)$, it follows that $\alpha_k(L)=\alpha_k(I)$ for all $0\leq k\leq n$.
Hence, $I=(u)$, as required.

$(2)\Leftrightarrow (3)$. It follows from the obviously fact: $\alpha_k(I)=\binom{n}{k}-\alpha_k(S/I)$.

$(2)\Rightarrow (4)$. Since $\alpha_k(I)=0$ for $0\leq k\leq m-1$, it follows that $\beta_k^n(I)=0$ for $0\leq k\leq m-1$.
Also, $\beta_m^n(I)=\alpha_m(I)=\binom{n-m}{0}=1$. Now, assume that $k>m$. From Lemma \ref{magic}, using the substitution $\ell=j-m$, 
we deduce that
\begin{align*}
 \beta_k^n(I)&=\sum_{j=m}^k (-1)^{k-j} \binom{n-j}{k-j}\binom{n-m}{j-m} = \sum_{\ell=0}^{k-m} (-1)^{k-m-\ell} \binom{(n-m)-\ell}{(k-m)-\ell}
\binom{n-m}{\ell} \\
& = \binom{(n-m)-(n-m)+k-m-1}{k-m}=\binom{k-m-1}{k-m}=0.
\end{align*}

$(4)\Rightarrow (2)$. Since $\alpha_k(I)=\sum_{j=0}^k \binom{n-j}{k-j}\beta_j^n(I)$ for all $0\leq k\leq n$, by (4) it
follows that $\alpha_k(I)=0$ for $k<m$ and $\alpha_k(I)=\binom{n-m}{k-m}$ for $k\geq m$. Note that $\binom{n-m}{k-m}=0$ for
$k<m$.

$(3)\Rightarrow (5)$. For any $0\leq k\leq n-1$, we have that \small
\begin{equation}\label{coocoo}
\beta_k^{n-1}(S/I)=\sum_{j=0}^k (-1)^{k-j} \binom{n-1-j}{k-j}\binom{n}{j} - \sum_{j=0}^k (-1)^{k-j} \binom{n-1-j}{k-j}\binom{n-m}{j-m}
\end{equation} \normalsize
Using the substitution $\ell=j-m$ we deduce that \small
\begin{equation}\label{coocoo_2}
\sum_{j=0}^k (-1)^{k-j} \binom{n-1-j}{k-j}\binom{n-m}{j-m}=\sum_{\ell=0}^{k-m} (-1)^{(k-m)-\ell} \binom{n-m-1-\ell}{k-m-\ell}\binom{n-m}{\ell}
\end{equation} \normalsize
Using Lemma \ref{magic}, from \eqref{coocoo} and \eqref{coocoo_2} we deduce that
\begin{align*}
 \beta_k^{n-1}(S/I)&=\binom{n-(n-1)+k-1}{k}-\binom{n-m-(n-m-1)+k-m-1}{k-m}= \\
& = \binom{k}{k}-\binom{k-m}{k-m}=\begin{cases} 1,& 0\leq k\leq m-1 \\
0,& m\leq k \leq n-1 \end{cases},
\end{align*}
as required.

$(5)\Rightarrow (3)$. If $0\leq k\leq m-1$, then 
$$\alpha_k(S/I)=\sum_{j=0}^k \binom{n-1-j}{k-j}\beta_j^{n-1}(S/I) = \sum_{j=0}^k \binom{n-1-j}{k-j} = \binom{n}{k}.$$
The last equality follows by induction on $k$. Indeed, if $k=0$, then there is nothing
to prove. Assuming, $k\geq 1$, by induction hypothesis we have:
\begin{align*}
 \sum_{j=0}^k \binom{n-1-j}{k-j} & =\binom{n-1}{k} + \sum_{j=1}^{k} \binom{n-1-j}{k-j} = 
\binom{n-1}{k} + \sum_{j=0}^{k-1} \binom{n-2-j}{k-1-j} \\
& = \binom{n-1}{k}+\binom{n-1}{k-1}=\binom{n}{k},
\end{align*}
as required. On the other hand, if $m\leq k\leq n$, then
$$\alpha_k(S/I)=\sum_{j=0}^{m-1} \binom{n-1-j}{k-j} = \binom{n}{k}-\binom{n-m}{k-m}.$$
Since $\binom{n-m}{k-m}=0$ for $k<m$, we are done.
\end{proof}

\begin{teor}\label{teo1}
Let $(0)\neq I\subset S$ be a proper monomial ideal. The following are equivalent:
\begin{enumerate}
\item[(1)] $I$ is principal.
\item[(2)] $\qdepth(I)=n$.
\item[(3)] $\qdepth(S/I)=n-1$.
\end{enumerate}
\end{teor}

\begin{proof}
Using polarization, we can assume that $I$ is squarefree. 

$(1)\Rightarrow(2)$. Since $I$ is principal, from Lemma \ref{liema} it follows that $\beta_k^n(I)\geq 0$ for all $0\leq k\leq n$.
Therefore, $\qdepth(I)=n$.

$(2)\Rightarrow (1)$. Since $I\neq (0)$ and $\qdepth(I)=n$, it follows that 
\begin{equation}\label{curcu}
1=\alpha_n(I)=\sum_{j=0}^n\beta_j^n(I)\text{ and }\beta_j^n(I)\geq 0\text{ for all }0\leq j\leq n.
\end{equation} 
Also, since $1\notin I$, we have $\beta_0^n(I)=\alpha_0(I)=0$. Hence, from \eqref{curcu} it follows that there
exists some integer $m$ with $1\leq m\leq n$ such that $\beta_m^n(I)=1$ and $\beta_k^n(I)=0$ for all $k\neq m$.
By Lemma \ref{liema}, we get the required result. 

$(1)\Rightarrow(3)$. Also, from Lemma \ref{liema} it follows that $\beta_k^{n-1}(S/I)\geq 0$ for all $0\leq k\leq n$ and thus
$\qdepth(S/I)\geq n-1$. On the other hand, as $I\neq (0)$, $\alpha_n(S/I)=0$ and thus $\qdepth(S/I)\leq n-1$. Hence, we are done.

$(3)\Rightarrow(1)$. Since $I\neq (0)$ and $\qdepth(S/I)=n-1$, it follows that
\begin{equation}\label{curcub}
0 < m:=\alpha_{n-1}(S/I)=\sum_{j=0}^{n-1}\beta_j^{n-1}(S/I)\text{ and }\beta_j^{n-1}(S/I)\geq 0\text{ for all }0\leq j\leq n-1.
\end{equation}
We claim that 
\begin{equation}\label{claimu}
\alpha_k(S/I)=\binom{n}{k}\text{ for all }0\leq k\leq m-1. 
\end{equation}
Indeed, if this is not the case, then there exists a squarefree monomial $v\in I$ with $\deg(v)=\ell<m$,
let's say $v=x_1x_2\cdots x_{\ell}$. 

It follows that $v_k:=x_1\cdots x_{k-1}x_{k+1}\cdots x_n\in I$ for all $\ell+1\leq k\leq n$ and
therefore $\alpha_{n-1}(S/I)\leq \ell<m$, a contradiction. Hence \eqref{claimu} is true.

From \eqref{claimu}, by straightforward computations we get that 
$$\beta_{k}^{n-1}(S/I)=1\text{ for all }0\leq k\leq m-1.$$
Hence, from \eqref{curcub} it follows that $\beta_k^{n-1}(S/I)=0$ for all $m\leq k\leq n-1$.
The required conclusion follows from Lemma \ref{liema}.
\end{proof}

\begin{lema}\label{patrat}
Let $I\subset S$ be a monomial ideal. Let $J:=(I,y_1,\ldots,y_m)\subset R:=S[y_1,\ldots,y_m]$, where $m\geq 1$.
We have that $$\qdepth(J)\geq \min\{\qdepth(S/I)+1,\qdepth(I)+m\}.$$
Moreover, if $\qdepth(I)\geq \qdepth(S/I)$, then $\qdepth(J)\geq \qdepth(R/J)+1$.
\end{lema}

\begin{proof}
Let $J_k:=IR+(y_1,\ldots,y_k)R\subset R$ for $0\leq k\leq m$.
From the filtration $IR=J_0\subset J_1\subset \cdots \subset J_m=J$, we deduce the following $K$-vector space decomposition
\begin{equation}\label{kuku1}
J=J_0 \oplus J_1/J_0 \oplus \cdots \oplus J_m/J_{m-1}.
\end{equation} 
From \eqref{kuku1} it follows that
\begin{equation}\label{kukuu1}
\beta_k^d(J)=\beta_k^d(J_0)+\beta_k^d(J_1/J_0)+\cdots+\beta_k^d(J_m/J_{m-1})\text{ for all }0\leq k\leq d\leq n+m.
\end{equation}
Note that $J_i/J_{i-1} \cong y_i(R/J_{i-1})$ for all $1\leq i\leq m$. Hence, from \cite[Lemma 2.13]{lucrare2} and
\cite[Theorem 2.22]{lucrare2} it follows that 
$$\qdepth(J_i/J_{i-1})=\qdepth(R/J_{i-1})=\qdepth((S/I)[y_i,\ldots,y_m])=$$
\begin{equation}\label{kuku2}
=\qdepth(S/I)+m-i+1\text{ for all }1\leq i\leq m.
\end{equation}
On the other hand, again from \cite[Lemma 2.10]{lucrare2} we have that
\begin{equation}\label{kuku3}
\qdepth(J_0)=\qdepth(IR)=\qdepth(I)+m.
\end{equation}
From \eqref{kukuu1}, \eqref{kuku2} and \eqref{kuku3} we get
$$\qdepth(J)\geq \min\{\qdepth(S/I)+1,\qdepth(I)+m\},$$
as required. Also, since $R/J\cong S/I$, the last assertion follows immediately.
\end{proof}

\begin{obs}\label{reducere}\rm
Let $I\subset S$ be a squarefree monomial ideal with $\qdepth(S/I)=q$. According to Theorem \ref{teo1}, if $q=n-1$, then $I$ is principal
and, therefore, $\qdepth(I)=n$. Hence, in order to compare $\qdepth(S/I)$ with $\qdepth(I)$ we can assume $q\leq n-2$. Another reduction
we can make is to assume that $I\subset \me^2$, where $\me= (x_1,\ldots,x_n)$. 

Indeed, if this is not the case, then, by reordering the variables, we can write 
$$I=(I',x_{m+1},\ldots,x_n)\text{ where }I'\subset S'=K[x_1,\ldots,x_m]\text{ with }I'\subset \me'^2\text{ and }$$ 
$\me'=(x_1,\ldots,x_m)S'$. According to Lemma \ref{patrat}, if $\qdepth(I')\geq \qdepth(S'/I')$, then 
$$\qdepth(I)\geq \qdepth(S/I)+1.$$
Note that $I\subset \me^2$ if and only if $\alpha_0(I)=\alpha_1(I)=0$, or, equivalently, $\alpha_0(S/I)=1$ and $\alpha_1(S/I)=n$.
\end{obs}

\section{Main results}

Given two positive integers $\ell,k$ there is a unique way
to expand $\ell$ as a sum of binomial coefficients, as follows
$$ \ell = \binom{n_k}{k}+\binom{n_{k-1}}{k-1}+\cdots+\binom{n_j}{j},\; n_k>n_{k-1}>\cdots>n_j\geq j\geq 1.$$
This expansion is constructed using the greedy algorithm, i.e. setting $n_k$ to be the maximal $n$ such that $\ell\geq\binom{n}{k}$,
replace $\ell$ with $\ell-\binom{n_k}{k}$ and $k$ with $k-1$ and repeat until the difference becomes zero.
We define $$\ell^{(k)}=\binom{n_k}{k+1}+\binom{n_{k-1}}{k}+\cdots+\binom{n_j}{j+1}.$$
We recall the famous Kruskal–Katona Theorem,
which gives a complete characterization of the $f$-vectors of simplicial complexes.

\begin{teor}(\cite[Theorem 2.1]{stanley})\label{krusk}
A vector $f=(1=f_{-1},f_0,f_1,\ldots,f_{d-1})$ is the $f$-vector of some simplicial complex $\Delta$ of dimension $d-1$
if and only if
$$ 0 < f_{i}\leq f_{i-1}^{(i-1)}\text{ for all }1\leq i\leq d-1. $$
\end{teor}

In the following, $(0)\neq I\subset S$ is a proper squarefree monomial ideal with $\qdepth(S/I)=q$, unless otherwise stated.
Moreover, according to Remark \ref{reducere}, we will assume $q\leq n-2$, $\alpha_0(I)=\alpha_1(I)=0$, $\alpha_0(S/I)=1$ and $\alpha_1(S/I)=n$.
In other words, the ideal $I$ is not principal and $I\subset \mathfrak m^2$.

The following result is a direct consequence of Theorem \ref{krusk} and of the interpretation of $I$ as the Stanley-Reisner ideal
associated to a simplicial complex: 

\begin{lema}\label{cord}
We have that:
$$ 0 < \alpha_k(S/I) \leq \alpha_{k-1}(S/I)^{(k-1)} \text{ for all }2\leq k\leq d.$$
In particular, if $\alpha_k(S/I)=\binom{n_k}{k} + \binom{n_{k-1}}{k-1}+\cdots+\binom{n_j}{j}$, where $2\leq k\leq d$ and $j\geq 1$, then
$$\alpha_{k-1}(S/I)\geq \binom{n_k}{k-1}+\binom{n_{k-1}}{k-2}+\cdots+\binom{n_j}{j-1}.$$
\end{lema}

\begin{proof}
Let $\Delta$ be the Stanley-Reisner simplicial complex associated to $I$. Let $d=\dim(\Delta)+1$.
It is easy to note that
$$\alpha_j(S/I)=f_{j-1}(S/I),\text{ for all }0\leq j\leq d.$$
The conclusion follows easily from Theorem \ref{krusk}.
\end{proof}



\begin{lema}\label{lem1}
Let $I\subset S$ be a squarefree monomial ideal. Then:
$$\beta_k^q(I)=\binom{n-q+k-1}{k}-\beta_k^q(S/I),\text{ for all }0\leq k\leq q\leq n.$$
\end{lema}

\begin{proof}
It follows from the obvious fact that $\alpha_j(I)=\binom{n}{j}-\alpha_j(S/I)$, for all $0\leq j\leq n$, and Lemma \ref{magic}.
\end{proof}

\begin{lema}\label{lem2}
The following are equivalent:
\begin{enumerate}
\item[(1)] $\qdepth(I)\geq \qdepth(S/I)+1$.
\item[(2)] $\beta_{k+1}^{q+1}(S/I)\leq \binom{n-q+k-1}{k+1},\text{ for all }0\leq k\leq q$.
\end{enumerate}
\end{lema}

\begin{proof}
First, note that $\beta_0^{q+1}(I)=0$. Also, from Lemma \ref{lem1} it follows that
$$\beta_{k+1}^{q+1}(I) = \binom{n-q+k-1}{k+1} - \beta_{k+1}^{q+1}(S/I),
\text{ for all }0\leq k\leq q.$$
The result follows immediately.
\end{proof}

In the following, our aim is to prove the condition (2) of Lemma \ref{lem1} for $q\leq 3$ and $0\leq k\leq q$.

\begin{lema}\label{lem3}
For any $0\leq k\leq q+1$, we have that 
$$\beta_k^k(S/I)\leq \alpha_k(S/I).$$
\end{lema}

\begin{proof}
First, note that $\beta_0^0(S/I)=\alpha_0(S/I)=1$. 

Assume that $1\leq k\leq q+1$.
From Corollary \ref{max} and \eqref{betak} it follows that 
$$\beta_{k}^{k}(S/I)=\alpha_k(S/I) - \beta_{k-1}^{k-1}(S/I)\leq \alpha_k(S/I),$$
as required.
\end{proof}

\begin{lema}\label{lem4}
If $q\geq 1$, then $\beta_2^{q+1}(S/I)\leq \binom{n-q}{2}$. 
\end{lema}

\begin{proof}
From \eqref{betak}, since $\alpha_0(S/I)=1$, it follows that
\begin{equation}\label{beta2}
\beta_2^{q+1}(S/I) = \alpha_2(S/I) - q \alpha_1(S/I) + \binom{q+1}{2}.
\end{equation}
If $\alpha_2(S/I)=\binom{n_2}{2}$ for some integer $2\leq n_2 \leq n$, then, from Lemma \ref{cord}, it follows
that $\alpha_1(S/I)\geq \binom{n_2}{1}=n_2$. Therefore, from \eqref{beta2} and Lemma \ref{magic} we get
\begin{equation}\label{caz21}
\beta_2^{q+1}(S/I) \leq \binom{n_2}{2} - qn_2 + \binom{q+1}{2} = \binom{n_2-q}{2} \leq \binom{n-q}{2}.
\end{equation}
If $\alpha_2(S/I)=\binom{n_2}{2}+\binom{n_1}{1}$ for $n>n_2>n_1\geq 1$, then, from Lemma \ref{cord},
it follows that $\alpha_1(S/I)\geq n_2+1$. As above, we get
\begin{equation}\label{caz22}
\beta_2^{q+1}(S/I) \leq \binom{n_2}{2} +n_2-1 - qn_2-q + \binom{q+1}{2} < \binom{n_2-q+1}{2}\leq \binom{n-q}{2}.
\end{equation}
From \eqref{caz21} and \eqref{caz22} we get the required result.
\end{proof}

\begin{lema}\label{lem5}
If $q=2$, then $\beta_3^3(S/I)\leq \binom{n-1}{3}$.
\end{lema}

\begin{proof}
For convenience, we denote $\alpha_j:=\alpha_j(S/I)$ for all $j$.
If $\alpha_2=\binom{n}{2}$, since $\alpha_3\leq \binom{n}{3}$, then 
$$\beta_3^3 = \alpha_3 - \alpha_2 + \alpha_1 - 1 \leq \binom{n}{3}-\binom{n}{2}+n-1 = \binom{n-1}{3},$$
and there is nothing to prove. Hence, we may assume $\alpha_2<\binom{n}{2}$.


Since $\alpha_0=1$, from \eqref{betak} we get
\begin{equation}\label{porc3}
\beta_3^{3}(S/I)= \alpha_3-\alpha_2+\alpha_1-1.
\end{equation}
From Lemma \ref{lem3} it follows that $\beta_3^3(S/I)\leq \alpha_3$.
Hence, if $\alpha_3\leq \binom{n-1}{3}$, then we are done. 

Suppose $\alpha_3>\binom{n-1}{3}$. From Lemma \ref{cord} we get
$\alpha_2 > \binom{n-1}{2}$ and therefore $\alpha_1=n$. We have two cases to consider:
\begin{enumerate}
\item[(i)] $\alpha_3=\binom{n-1}{3}+\binom{n_2}{2}$ with $2\leq n_2\leq n-2$. From
Lemma \ref{cord} we get $\alpha_2\geq \binom{n-1}{2}+n_2$. Thus, from \eqref{porc3} it follows that
$$\beta_3^{3}(S/I)\leq \binom{n-1}{3}+\binom{n_2}{2} - \binom{n-1}{2} - n_2 + n-1.$$
Hence, it is enough to show that 
$$\binom{n_2}{2}-n_2 = \frac{n_2(n_2-3)}{2} \leq \binom{n-1}{2}-(n-1) = \frac{(n-1)(n-4)}{2},$$
which is clear as $n_2\leq n-2$.

\item[(ii)] $\alpha_3=\binom{n-1}{3}+\binom{n_2}{2}+\binom{n_1}{1}$ with $1\leq n_1 < n_2\leq n-2$.
From Lemma \ref{cord} it follows that $\alpha_2\geq \binom{n-1}{2}+n_2+1$. Thus, from \eqref{porc3},
as in the case (i), it suffices to show that
$$\binom{n_2}{2}-n_2+n_1-1 = \frac{n_2(n_2-3)}{2} + n_1-1 \leq \frac{(n-1)(n-4)}{2}.$$
Since $n_1\leq n-3$, in order to prove the above condition, it is enough to show that
$$\frac{n_2(n_2-3)}{2} \leq \frac{(n-3)(n-4)}{2},$$
which is true, since $n_2\leq n-2$ and $(n-2)(n-5)\leq (n-3)(n-4)$ for all $n\geq 3$.
\end{enumerate}
Thus, the proof is complete.
\end{proof}

\begin{lema}\label{lucky}
If $q=3$, then 
$$\beta_3^4(S/I)\leq \binom{n-2}{3}\text{ and }\beta_4^4(S/I)\leq \binom{n-1}{4}.$$
\end{lema}

\begin{proof}
For convenience, we denote 
$$\beta_k^d=\beta_k^d(S/I)\text{ and }\alpha_k=\alpha_k(S/I), \text{ for all }0\leq k\leq d\leq n.$$
From hypothesis we have $\alpha_0=1$, $\alpha_1=n$ and $n\geq q+2=5$; see Remark \ref{reducere}. It follows that
$$\beta_2^3=\alpha_2-2n+3\geq 0,\;\beta_3^3=\alpha_3-\alpha_2+n-1\geq 0.$$
Therefore, we have that
\begin{equation}\label{coor}
\alpha_2\geq 2n-3\text{ and }\alpha_3\geq \alpha_2-n+1.
\end{equation}
Also, at least one of $\beta_k^4$, with $1\leq k\leq 4$, is negative. Henceforth, we consider the cases:
\begin{enumerate}
\item[(i)] If $\beta_1^4=\alpha_1-4<0$, then $n=\alpha_1 \leq 3$, a contradiction with the fact that $n\geq 5$.

\item[(ii)] If $\beta_2^4=\alpha_2-3\alpha_1+6<0$, then, from \eqref{coor}, it follows that
\begin{equation}\label{coor3}
3n-7 \geq \alpha_2 \geq 2n-3.
\end{equation}
Let $n_2$ such that 
\begin{equation}\label{n2}
\binom{n_2-1}{2}<3n-7\leq \binom{n_2}{2}.
\end{equation}
Since $\alpha_2\leq 3n-7$, from Lemma \ref{cord} it follows that
\begin{equation}\label{coor4}
\alpha_3 \leq \binom{n_2}{3}\text{ and } \alpha_4\leq \binom{n_2}{4}.
\end{equation}
From \eqref{coor3} and \eqref{coor4} and the choice of $n_2$ we get
\begin{equation}\label{coor5}
 \beta_3^4 = \alpha_3 - 2 \alpha_2 + 3 \alpha_1 - 4 \leq \binom{n_2}{3}-n+2.
\end{equation}
Hence, in order to show that $\beta_3^4 \leq \binom{n-2}{3}$, it suffices to prove that \small
\begin{equation}\label{coor6}
\binom{n_2}{3} \leq \binom{n-2}{3} + (n-2). 
\end{equation} \normalsize
Note that $n_2\leq n$, since $3n-7<\binom{n}{2}$.
If $n_2\leq n-2$, then \eqref{coor6} obviously holds and, moreover, from Lemma \ref{lem3} and \eqref{coor4} it follows that
$$\beta_4^4 \leq \alpha_4 \leq \binom{n_2}{4}\leq \binom{n-2}{4},$$
as required. Therefore, we can assume that $n_2\in \{n-1,n\}$. We consider two subcases:
\begin{enumerate}
\item[(i.a)] $n_2=n-1$. From \eqref{coor4} and Lemma $3.3$ we have
$$\beta_4^4 \leq \alpha_4 \leq \binom{n-1}{4},$$
as required. From \eqref{n2} it follows that
$$n^2-5n+8 \leq 6n-14 \leq n^2-3n+2.$$
An easy calculation shows that these inequalities hold only for $n\in \{7,8\}$.
\begin{itemize}
\item If $n=7$, then $n_2=6$. Also, from \eqref{coor3} we have $11\leq \alpha_2\leq 14$.
      Using Lemma \ref{cord} we deduce Table 1:
			\begin{table}[htb]
      \centering
      \caption{}
      \begin{tabular}{|l|l|l|l|l|l|l|}
      \hline
      $\alpha_2$       & 11 & 12 & 13 & 14 \\ \hline
      $\max(\alpha_3)$ & 10 & 11 & 13 & 16 \\ \hline
			\end{tabular}
      \end{table}
			
			From the table, we deduce that $\max(\alpha_3-2\alpha_2)=-12$. It follows that
      $$\beta_3^4\leq -12 + 21 - 4 = 5 \leq \binom{5}{3}.$$			
\item If $n=8$, then $n_2=7$. Also, from \eqref{coor3} we have $13\leq \alpha_2\leq 17$.
      Using Lemma \ref{cord} we deduce Table 2:
      \begin{table}[htb]
      \centering
      \caption{}
      \begin{tabular}{|l|l|l|l|l|l|l|l|}
      \hline
      $\alpha_2$       & 13 & 14 & 15 & 16 & 17 \\ \hline
      $\max(\alpha_3)$ & 13 & 16 & 20 & 20 & 21 \\ \hline
      \end{tabular}
      \end{table}
			
      From the table, we deduce that $\max(\alpha_3-2\alpha_2)=-10$. It follows that
      $$\beta_3^4\leq -10 + 24 - 4 = 10 \leq \binom{6}{3}.$$					
\end{itemize}
\item[(i.b)] $n_2=n$. From \eqref{n2} it follows that
$$n^2-3n \leq 6n-14 \leq n^2-n.$$
These conditions hold if and only if $n\in \{2,3,4,5,6,7\}$. Since $n\geq 5$, we have in fact
$n\in \{5,6,7\}$.
\begin{itemize}
\item $n=5$. From \eqref{coor3} it follows that $7\leq \alpha_2\leq 8$. Lemma \ref{cord} implies that
     $$\alpha_2=7\Rightarrow \alpha_3\leq 4\text{ and }\alpha_2=8\Rightarrow \alpha_3\leq 5.$$
		 It follows immediately that $$\beta_3^4=\alpha_3-2\alpha_2+3\alpha_1-4=\alpha_3-2\alpha_2+11\leq 1=\binom{5-2}{3},$$
		 as required. On the other hand, from Lemma \ref{lem3} we have $\beta_4^4\leq \alpha_4$. 
		 If $\alpha_4\geq 2$, then $\alpha_3\geq 7$, a contradiction as $\alpha_3\leq 5$. Thus $\beta_4^4\leq 1=\binom{5-1}{4}$,
		 as required.
\item $n=6$. From \eqref{coor3} it follows that $9\leq \alpha_2\leq 11$. As in the subcase (i.a), we deduce
     that $\max\{\alpha_3-2\alpha_2\}=-10$. Since $\alpha_1=6$, we have
		 $$\beta_3^4\leq -10+18-4 = 4 = \binom{6-2}{3},$$
		 as required. On the other hand, since $\alpha_3\leq 10=\binom{5}{3}$, from Lemma \ref{cord} it follows that $\alpha_4\leq \binom{5}{4}=5$.
		 Thus, from Lemma \ref{lem3} it follows that $$\beta_4^4\leq\alpha_4\leq 5 =\binom{6-1}{4}.$$
\item $n=7$. From \eqref{coor3} it follows that $11\leq \alpha_2\leq 14$. As in the subcase (i.a), we deduce
     that $\max\{\alpha_3-2\alpha_2\}=-12$. Since $\alpha_1=7$, we have
		 $$\beta_3^4\leq -12+21-4 = 5 \leq \binom{7-2}{3},$$
		 as required. On the other hand, since $\alpha_3\leq 16=\binom{5}{3}+\binom{4}{2}$, Lemma \ref{cord} implies
		 $\alpha_4\leq \binom{5}{4}+\binom{4}{3}=9$.
		 Thus, from Lemma \ref{lem3} it follows that $$\beta_4^4\leq\alpha_4\leq 9 < 15=\binom{7-1}{4}.$$
\end{itemize}
\end{enumerate}

\item[(iii)] If $\beta_3^4<0$, then, in particular, $\beta_3^4\leq \binom{n-2}{3}$. If $\alpha_4\leq \binom{n-1}{4}$, then, from Lemma \ref{lem3},
             it follows that $\beta_4^4\leq \binom{n-1}{4}$ and there is nothing to prove. Hence, we
		         may assume that $\alpha_4 \geq \binom{n-1}{4}+\binom{3}{3}$. From Lemma \ref{cord} it follows that:
		         \begin{equation}\label{coco2}
		         \alpha_3 \geq \binom{n-1}{3} + \binom{3}{2}\text{ and }\alpha_2\geq \binom{n-1}{2}+ \binom{3}{1}.
		         \end{equation}
             Since $\beta^4_2,\beta^3_3\geq 0$, it follows that
		         \begin{equation}\label{coco1}
		         \alpha_2\geq 3n-6\text{ and }\alpha_2-n+1 \leq \alpha_3 \leq 2\alpha_2-3n+3.
             \end{equation}
             From the fact that $n\geq 5$, \eqref{coco2} and \eqref{coco1} it follows that
		         \begin{equation}\label{coco3}
		         \binom{n-1}{3}+3 \leq \alpha_3 \leq 2\alpha_2-3n+3 \leq 2\binom{n}{2}-3n+3=(n-1)(n-3),
		         \end{equation}
		         from which we deduce that $5\leq n\leq 7$. We have the subcases:
		\begin{enumerate}

\item[(iii.a)] $n=5$. From \eqref{coco3} it follows that $7\leq \alpha_3 \leq 8$. Since $\alpha_3 \leq 8 = \binom{4}{3}+\binom{3}{2}+\binom{1}{1}$,
                      from Lemma \ref{cord} it follows that $\alpha_4\leq 2$. Hence
                      $$\beta^4_4 = \alpha_4-\alpha_3+\alpha_2 - 4 \leq 2-7+10-4= 1=\binom{5-1}{4}.$$
											
\item[(iii.b)] $n=6$. From \eqref{coco3} it follows that $13 \leq \alpha_3 \leq 15$.
               From Lemma \ref{cord}, we deduce the following:

\begin{table}[htb]
\centering
\caption{}
\begin{tabular}{|l|l|l|l|l|l|}
\hline
$\alpha_3$       &  13 &  14 & 15 \\ \hline
$\max(\alpha_4)$ &   6  & 6 &  7 \\ \hline
\end{tabular}
\end{table}
From the above table, and the fact that $\alpha_2\leq \binom{6}{2}=15$, it follows that 
$$\beta^4_4 = \alpha_4-\alpha_3+\alpha_2 - 5 \leq -7+15-5 = 3 < 5 = \binom{6-1}{4}.$$

\item[(iii.c)] $n=7$. From \eqref{coco3} it follows that $23 \leq \alpha_3 \leq 24$.
               From Lemma \ref{cord}, we deduce the following:

\begin{table}[htb]
\centering
\caption{}
\begin{tabular}{|l|l|l|l|l|}
\hline
$\alpha_3$       &  23 &  24 \\ \hline
$\max(\alpha_4)$ & 16  & 17  \\ \hline
\end{tabular}
\end{table}

From the above table, and the fact that $\alpha_2\leq \binom{7}{2}=21$, it follows that 
$$\beta^4_4 = \alpha_4-\alpha_3+\alpha_2 - 6 \leq -7+21-6=8 \leq 15=\binom{7-1}{4}.$$
\end{enumerate}

\item[(iv)] If $\beta_4^4<0$, then, in particular, $\beta_4^4\leq \binom{n-1}{4}$. Also, since $\beta_4^4=\alpha_4-\beta_3^3$,
     we have 
		\begin{equation}\label{cuc1}
		1\leq \alpha_4+1\leq \beta_3^3\leq \alpha_3\leq \binom{n}{3}\text{ and thus }\alpha_3\geq\alpha_2-n+2.
		\end{equation}
    In order to avoid the previous cases,
     we can assume that $\beta^4_2,\beta^4_3\geq 0$ and therefore
		 \begin{equation}\label{cuc2}
		 \alpha_2\geq 3n-6,\;\alpha_3\geq 2\alpha_2-3n+4.
		 \end{equation}
     From \eqref{cuc2} it follows that
		\begin{equation}\label{cuc3}
		\beta_3^4 = \alpha_3 -2\alpha_2+3n-4\leq \alpha_3-2(3n-6)+3n-4=\alpha_3-3n+8.
		\end{equation}
		
		If $\alpha_3\leq \binom{n-2}{3}+3n-8$, then there is nothing to prove. Assume that 
		\begin{equation}\label{alfa3}
		\alpha_3\geq \binom{n-2}{3}+3n-7.
		\end{equation}
		If $n=5$, then \eqref{alfa3} implies $\alpha_3\geq 9 = \binom{4}{3}+\binom{3}{2}+\binom{2}{1}$ and, therefore, by Lemma \ref{cord},
		we have $\alpha_2 \geq 10$. It follows that
		$$\beta_3^4 = \alpha_3 - 2\alpha_2+3\cdot 5-4 \leq \alpha_3 - 9 \leq \binom{5}{3}-9=1 = \binom{5-2}{3},$$
		as required. Therefore, we may assume $n\geq 6$.
					
		
		
		In order to prove that $\beta_3^4\leq \binom{n-2}{3}$ it is enough to prove that
		\begin{equation}\label{dorinta1}
				\alpha_3-2\alpha_2 \leq \binom{n-2}{3}-3n+4 = \frac{n(n-1)(n-8)}{6}.
		\end{equation}		
		We consider two subcases:
		\begin{enumerate}
		\item[(iv.a)] $\alpha_3>\binom{n-1}{3}$. From Lemma \ref{cord} it follows that $\alpha_2>\binom{n-1}{2}$.
		              If $\alpha_3=\binom{n}{3}$, then $\alpha_2=\binom{n}{2}$ and
		              therefore $\alpha_3-2\alpha_2 = \frac{n(n-1)(n-8)}{6}$, hence \eqref{dorinta1} holds.
									
									If $\alpha_2=\binom{n-1}{2}+n_1$ for some $1\leq n_1\leq n-2$, then $\alpha_3\leq \binom{n-1}{3}+\binom{n_1}{2}$. Thus
		              $$\alpha_3-2\alpha_2 \leq \binom{n-1}{3}-2\binom{n-1}{2}+\binom{n_1}{2}-2n_1=$$
									\begin{equation}\label{aoleu}
		              =\frac{(n-1)(n-2)(n-9)}{6}+\frac{n_1(n_1-5)}{2}.
		              \end{equation}
		              Since $n\geq n_1+2$, it is clear that $\frac{n_1(n_1-5)}{2}\leq \frac{(n-2)(n-7)}{2}$. Thus, \eqref{aoleu} implies
		              $$\alpha_3-2\alpha_2 \leq \frac{(n-1)(n-2)(n-9)}{6}+\frac{3(n-2)(n-7)}{6}=$$
									\begin{equation}\label{vaivaleu}
		              =\frac{(n-2)(n^2-7n-12)}{6}.
		              \end{equation}
		              Since $n(n-1)(n-8)-(n-2)(n^2-7n-12)=6(n-4)>0$, from \eqref{vaivaleu} it follows that \eqref{dorinta1} is satisfied.
									
    \item[(iv.b)] $\alpha_3\leq \binom{n-1}{3}$. From \eqref{alfa3} it follows that 
		              $$\binom{n-1}{3}-\binom{n-2}{3}-3n+7\geq 0 \Leftrightarrow n^2-11n+20\geq 0.$$
									Since $n\geq 5$, the above inequality implies	$n\geq 9$. 
									
									From \eqref{alfa3} we deduce that																
									$$\alpha_3\geq \binom{n-2}{3}+3\cdot 9 - 7 = \binom{n-2}{3} + \binom{6}{2} + \binom{5}{1}.$$
									From \eqref{cord} it follows that $\alpha_2\geq \binom{n-2}{2}+7$. Therefore, we get
									$$\alpha_3-2\alpha_2 \leq \binom{n-1}{3} - (n-2)(n-3) - 14.$$
									It is easy to check that
									$$\binom{n-1}{3} < (n-2)(n-3)+14+\frac{n(n-1)(n-8)}{6}.$$
									Thus \eqref{dorinta1} holds.
		\end{enumerate}
\end{enumerate}
\end{proof}

\begin{teor}\label{teo2}
Let $(0)\neq I\subset S$ be a proper squarefree monomial ideal with $q=\qdepth(S/I)\leq 3$.
Then $$\qdepth(I)\geq \qdepth(S/I)+1.$$
\end{teor}

\begin{proof}
If $I$ is principal, then, according to Theorem \ref{teo1}, there is nothing to prove.
Also, using the argument from Remark \ref{reducere}, we can assume that $I\subset \me^2$. (We need this
assumption in order to apply several of the previous lemmas.)

If $q=0$, then there is nothing to prove, hence we may assume $q\geq 1$. From Lemma \ref{lem3}, it
is enough to show that 
\begin{equation}\label{desire}
\beta^{q+1}_{k+1}(S/I)\leq \binom{n-q+k-1}{k+1}\text{ for all }0\leq k\leq q.
\end{equation}
Since $\alpha_0(S/I)=1$, we have that:
$$\beta^{q+1}_1(S/I)=\alpha_1(S/I)-(q+1)\alpha_0(S/I)\leq n-q-1 = \binom{n-q+0-1}{1},$$
and thus \eqref{desire} holds for $k=0$. Also, from Lemma \ref{lem4} it follows that \eqref{desire} holds for $k=1$.
In particular, the case $q=1$ is proved.

Similarly, since \eqref{desire} holds for $k\in\{0,1\}$, 
the case $q=2$ follows from Lemma \ref{lem5} and the case $q=3$ follows from Lemma \ref{lucky}.
\end{proof}

\begin{cor}\label{cor2}
Let $I\subset S=K[x_1,\ldots,x_n]$ be a squarefree monomial ideal. If $n\leq 5$, then 
$$\qdepth(I)\geq \qdepth(S/I)+1.$$
\end{cor}

\begin{proof}
Let $q=\qdepth(S/I)$. If $n\leq 4$, then $q\leq 3$ and the conclusion follows from Theorem \ref{teo2}.
If $n=5$, then $q\leq 4$. If $q=4$, then, according to Theorem \ref{teo1}, $I$ is principal and, moreover,
$\qdepth(I)=5$. Also, if $q\leq 3$, then we are done by Theorem \ref{teo2}.
\end{proof}

Let $I\subset S$ be a squarefree monomial ideal with $\qdepth(S/I)\geq 4$. 
If $S/I$ is not Cohen-Macaulay, then the inequality 
$\qdepth(I)\geq \qdepth(S/I)+1$ does not necessarily hold, as the following example shows:

\begin{exm}\label{patru}\rm
We consider the ideal
$$I=(x_1x_2,x_1x_3,x_1x_4,x_1x_5x_6)\subset S=K[x_1,x_2,\ldots,x_6].$$
By straightforward computations, we get
$$\alpha(S/I)=(1,6,12,10,5,1,0)\text{ and }\alpha(I)=(0,0,3,10,10,5,1).$$
Also, from \eqref{betak}, we get 
$$\beta^4(S/I)=(1,2,0,0,2),\;\beta^5_2(S/I)=-2<0,\;\beta^4(I)=(0,0,3,4,3)\text{ and }\beta_4^5(I)=-1.$$
Hence, $\qdepth(S/I)=\qdepth(I)=4$. 
\end{exm}

Similarly to Lemma \ref{lem2}, we have the following:

\begin{lema}\label{lem22}
Let $I\subset S$ be a proper squarefree monomial ideal with \linebreak $\qdepth(S/I)=q$.
The following are equivalent:
\begin{enumerate}
\item[(1)] $\qdepth(I)\geq \qdepth(S/I).$
\item[(2)] $\beta_{k}^{q}(S/I)\leq \binom{n-q+k-1}{k},\text{ for all }1\leq k\leq q$.
\end{enumerate}
\end{lema}

\begin{prop}\label{cook}
For any squarefree monomial ideal $I\subset S$, condition (2) from Lemma \ref{lem22} holds for $k\in\{1,2\}$.
\end{prop}

\begin{proof}
If $q\geq 1$, then $\beta_1^q(S/I)=\alpha_1(S/I)-q\leq n-q$, as required. If $q\geq 2$, then, according to
Lemma \ref{lem4}, applied for $q-1$, it follows that $\beta_2^q(S/I)\leq \binom{n-q+1}{2}$, as required.
\end{proof}

As Example \ref{minus} shows, condition (2) of Lemma \ref{lem22} does not hold in general for $k=3$.

\begin{teor}\label{t3}
If $I\subset S$ is a squarefree monomial with $\qdepth(S/I)=4$, then 
$$\qdepth(I)\geq 4.$$
\end{teor}

\begin{proof}
An in the proof of Theorem \ref{teo2}, we can assume that $I$ is not principal and $I\subseteq \me^2$. In particular, $n\geq 4+2=6$.
For convenience, we denote $\alpha_k=\alpha_k(S/I)$ and $\beta_k^s=\beta_k^s(S/I)$ for all $0\leq k\leq s\leq n$.

Since in the proof of the Case 4 of Lemma \ref{lucky}
we have $\beta_1^4,\beta_2^4,\beta_3^4\geq 0$ and we don't use any assumption on $\alpha_4$ and on
$\beta_4^4$, we can apply the same arguments in order to conclude that 
$$\beta_3^4 \leq \binom{n-2}{3}.$$
Hence, in order to complete the proof, we have to show that 
$\beta_4^4\leq \binom{n-1}{4}$.
We will use similar methods as in the
proof of Case 3 of Lemma \ref{lucky} with the difference that, in that case, we had $\beta_3^4<0$. 
From Proposition \ref{cook} and the fact that $\qdepth(S/I)=4$, we have 
\begin{equation}\label{cocos1}
0\leq \beta_1^4 = \alpha_1 - 4 = n-4\text{ and }0\leq \beta_2^4 = \alpha_2 - 3\alpha_1 + 6 \leq \binom{n-3}{2}.
\end{equation}
Also, we have that
\begin{equation}\label{cocos2}
0\leq \beta_3^4 = \alpha_3 - 2\alpha_2 + 3n-4 \leq \binom{n-2}{3}\text{ and }\beta_4^4=\alpha_4-\alpha_3+\alpha_2-n+1\geq 0
\end{equation}
On the other hand, since $\beta_4^4=\alpha_4-\beta_3^3$ and $\beta_3^3\geq 0$, we can assume that 
$$\alpha_4\geq \binom{n-1}{4} + 1 = \binom{n-1}{4}+\binom{3}{3},$$ otherwise
there is nothing to prove. Therefore, from Lemma \ref{cord}, we have $\alpha_3\geq \binom{n-1}{3}+3$. 
Since $\binom{n}{4}-\binom{n}{3}+\binom{n}{2}-n+1=\binom{n-1}{4}$ and 
$\alpha_2\leq \binom{n}{2}$, in order to complete the proof it suffices to 
show that 
\begin{equation}\label{pohta}
\alpha_4-\alpha_3\leq \binom{n}{4}-\binom{n}{3}=\frac{1}{24}n(n-1)(n-2)(n-7).
\end{equation}
We consider two cases:
\begin{enumerate}
\item[(i)] $\alpha_3=\binom{n-1}{3} + \binom{n_2}{2}$ with $n-1>n_2\geq 3$. From Lemma \ref{cord}, we have
           $\alpha_4\leq \binom{n-1}{4} + \binom{n_2}{3}$ and thus
					 $$\alpha_4 - \alpha_3 \leq \binom{n-1}{4} - \binom{n-1}{3} +\binom{n_2}{3} - \binom{n_2}{2} = $$
					 \begin{equation}\label{cucurigu}
					 = \frac{1}{24}(n-1)(n-2)(n-3)(n-8)+\frac{1}{6}n_2(n_2-1)(n_2-5).
					 \end{equation}
					 Since $n\geq 6$, if $n_2\leq 5$, then \eqref{cucurigu} implies \eqref{pohta} and we are done.
					 Now, assume that $n_2\geq 6$ and thus $n\geq 8$. Hence, from \eqref{cucurigu} it follows that					
					 \begin{align*}
					 & \alpha_4-\alpha_3 \leq \frac{1}{24}(n-1)(n-2)(n-3)(n-8)+\frac{1}{6}(n-2)(n-3)(n-7)=\\
					 & = \frac{1}{24}(n-2)(n-3)(n^2-5n-20)=\frac{1}{24}(n-2)(n^3-8n^2-5n+60)\leq \\
					 & \leq \frac{1}{24}(n-2)(n^3-8n^2+7n) = \binom{n}{4}-\binom{n}{3},
					 \end{align*}
					 and thus we are done.
\item[(ii)] $\alpha_3=\binom{n-1}{3} + \binom{n_2}{2}+\binom{n_1}{1}$ with $n-1>n_2\geq 3$ and $n_2>n_1\geq 1$.
           From Lemma \ref{cord}, we have
           $\alpha_4\leq \binom{n-1}{4} + \binom{n_2}{3}+\binom{n_1}{2}$ and thus
					 $$\alpha_4 - \alpha_3 \leq \binom{n-1}{4} - \binom{n-1}{3} +\binom{n_2}{3} - \binom{n_2}{2} + \binom{n_1}{2}-n_1 = $$
					 \begin{equation}\label{cucurigu2}
					 = \frac{1}{24}(n-1)(n-2)(n-3)(n-8)+\frac{1}{6}n_2(n_2-1)(n_2-5)+\frac{1}{2}n_1(n_1-3).
					 \end{equation}
					 If $n_2\leq 4$, then $n_1\leq 3$ and thus, since $n\geq 6$, \eqref{cucurigu2} implies \eqref{pohta}.
					 If $n_2=5$, then $n_1\leq 4$ and $n\geq 7$ and, again, \eqref{cucurigu2} implies \eqref{pohta}. 
					
					 If $n_2\geq 6$, then $n\geq 8$ and from \eqref{cucurigu2} we get \small
					 \begin{align*} 
					  \binom{n}{4}-\binom{n}{3} - (\alpha_4-\alpha_3)  \geq & \binom{n}{4}-\binom{n}{3} -  \binom{n-1}{4} + \binom{n-1}{3} 
					   - \binom{n-2}{3} \\ 
						& + \binom{n-2}{2} - \binom{n-3}{2} + \binom{n-3}{1} =n-4,
					 \end{align*} \normalsize
					 and thus the conclusion follows from \eqref{pohta}.
\end{enumerate}
\end{proof}

\begin{cor}\label{cor3}
Let $I\subset S=K[x_1,\ldots,x_6]$ be a squarefree monomial ideal. Then $$\qdepth(I)\geq \qdepth(S/I).$$
\end{cor}

\begin{proof}
Let $q=\qdepth(S/I)$. If $q=5$, then, according to Theorem \ref{teo1}, $I$ is principal and we have $\qdepth(I)=6$.
If $q\leq 4$, then the conclusion follows from Theorems \ref{teo2} and \ref{t3}.
\end{proof}

We may ask if the inequality $\qdepth(I)\geq \qdepth(S/I)$ holds in general. The answer is negative:

\begin{exm}\label{minus}\rm
Consider the ideal 
$$I=(x_1)\cap(x_2,x_3,\ldots,x_{13})\subset K[x_1,\ldots,x_{13}].$$
It is easy to check that $\alpha_0(S/I)=1$, $\alpha_1(S/I)=13$ and 
$\alpha_k(S/I)=\binom{12}{k}$ for all $2\leq k\leq 13$. By straightforward computations, 
we get $\qdepth(S/I)=8$. On the other hand, we have $\alpha_0(I)=\alpha_1(I)=0$, $\alpha_2(I)=12$ and $\alpha_3(I)=66$.
Since $\beta^{8}_3(I)=66-6\cdot 12<0$ it follows that $\qdepth(I)<8$. In fact, we have $\qdepth(I)=7$.

Note that, according to Lemma \ref{lem1}, $\beta^{8}_3(I)<0$ is equivalent to $\beta^8_3(S/I)>\binom{13-8+3-1}{3}=\binom{7}{3}$.
Hence, the condition (2) from Lemma \ref{lem22} does not hold.

\end{exm}

\begin{exm}\label{minus2}\rm
Consider the ideal 
$$I=(x_1)\cap (x_ix_j\;:\;2\leq i < j \leq 14)\subset K[x_1,\ldots,x_{14}].$$
It is easy to check that $\alpha_0(S/I)=1$, $\alpha_{14}(S/I)=0$, $\alpha_1(S/I)=14$, $\alpha_2(S/I)=\binom{14}{2}$ and 
$\alpha_k(S/I)=\binom{13}{k}$ for all $3\leq k\leq 13$. 

By straightforward computations, 
we get $\qdepth(S/I)=7$. On the other hand, we have $\alpha_0(I)=\alpha_1(I)=\alpha_2(I)=0$, $\alpha_3(I)=\binom{13}{2}$ and
$\alpha_4(I)=\binom{13}{3}$. Since $\beta^7_4(I)=\alpha_4(I)-4\alpha_3(I)<0$ it follows that $\qdepth(I)<7$. In fact, it is easy
to check that $\qdepth(I)=6$. 
Again, $\beta^7_4(I)<0$ is equivalent to $\beta^7_4(S/I)> \binom{14-7+4-1}{4}=\binom{10}{4}$.
\end{exm}

\begin{exm}\label{minus3}\rm
Let $n=10$ and $\alpha_j=\binom{10}{j}$ for $0\leq j\leq 4$.
We also let $\alpha_j=\binom{9}{j}+\binom{8}{j-1}+\binom{6}{j-2}$ for $5\leq j\leq 7$. Hence
$\alpha=(1,10,45,120,197,216,155,70)$. From Kruskal-Katona Theorem there exists a squarefree monomial $I\subset S=K[x_1,\ldots,x_{10}]$
such that $\alpha_j=\alpha_j(S/I)$ for all $0\leq j\leq 7$. By straightforward computations, we have $\qdepth(S/I)=7$ and, moreover,
$$\beta_5^7(S/I)=24>\binom{10-7+5-1}{5}=\binom{7}{5}=21.$$
This implies that $\qdepth(I)\leq 6$. In fact, we have $\qdepth(I)=6$.
\end{exm}

The above examples yield us to propose the following conjectures:

\begin{conj}\label{conj1}
If $I\subset S$ is a squarefree monomial ideal with $\qdepth(S/I)\leq 6$, then
$$\qdepth(I)\geq\qdepth(S/I).$$
\end{conj}

\begin{conj}\label{conj2}
Let $I\subset S=K[x_1,\ldots,x_n]$ be a squarefree monomial. If $n\leq 9$, then
$$\qdepth(I)\geq\qdepth(S/I).$$
\end{conj}

In the next section we will give a partial answer to these conjectures.

\section{The case $\qdepth(S/I)=5$.}

Let $(0)\neq I\subset S$ be a proper squarefree monomial ideal with $q:=\qdepth(S/I)\geq 5$. We will also assume that $I$ is not principal
and $I\subset \mathfrak m^2$. In particular, we have $n\geq q+2$.

\begin{lema}\label{b3q}
Suppose that $q\in\{5,6,7\}$. We have that:
$$\beta^q_3(S/I)\leq \binom{n-q+2}{3}.$$
\end{lema}

\begin{proof}
We denote $\alpha_j:=\alpha_j(S/I)$ for all $j\geq 0$, and $\beta_k^q=\beta_k^q(S/I)$ for all $0\leq k\leq q$.
From \eqref{betak} we have that 
\begin{equation}\label{eco1}
\beta_3^q=\alpha_3-(q-2)\alpha_2+\binom{q-1}{2}\alpha_1-\binom{q}{3}.
\end{equation}
Since, from hypothesis, we have that 
$$\alpha_0=1,\; \alpha_1=n\text{ and }\binom{n}{3}-(q-2)\binom{n}{2}+\binom{q-1}{2}n-\binom{q}{3}=\binom{n-q+2}{3},$$
in order to complete the proof it is enough to show that
\begin{equation}\label{eco2}
\alpha_3-(q-2)\alpha_2 \leq \binom{n}{3}-(q-2)\binom{n}{2}=\frac{n(n-1)(n-3q+4)}{6}.
\end{equation}
Since $q=\qdepth(S/I)$, it follows that
\begin{equation}\label{cond}
\beta_2^q =\alpha_2-(q-1)\alpha_1+\binom{q}{2}=\alpha_2-n(q-1)+\binom{q}{2}\geq 0,\;\alpha_2\geq \frac{(2n-q)(q-1)}{2}.
\end{equation}
We consider the function 
\begin{equation}\label{ecos5}
f(x)=\binom{x}{3}-(q-2)\binom{x}{2}=\frac{x(x-1)(x-3q+4)}{6},\text{ where }x\geq 0.
\end{equation}
The derivative of $f(x)$ is the quadratic function
$$f'(x)=\frac{1}{2}x^2-(q-1)x+\frac{q}{2}-\frac{2}{3}.$$
Since the discriminant of $f'(x)$ is $\Delta=q^2-3q+\frac{7}{3}$, it follows that 
$$f'(x)\leq 0\text{ for }x\in \left[q-1-\sqrt{q^2-3q+\frac{7}{3}},q-1+\sqrt{q^2-3q+\frac{7}{3}}\right].$$
Therefore
\begin{equation}\label{eco5}
\begin{split}
& f(0)=f(1)=0,\;f(2q-3)=f(2q-2),\\
& f(x) \searrow \text{ on }[1,2q-3]\text{ and }f(x) \nearrow \text{ on }[2q-2,\infty).
\end{split}
\end{equation}
Assume that $\alpha_2=\binom{n_2}{2}$ for some integer $n_2 \geq 2$. 
From Lemma \ref{cord} it follows that $\alpha_3\leq \binom{n_2}{3}$ and, therefore
\begin{equation}\label{eco3}
\alpha_3 - (q-2)\alpha_2 \leq \binom{n_2}{3}-(q-2)\binom{n_2}{2}=f(n_2).
\end{equation}
If $n_2=n$, then from \eqref{eco3} it follows that equation \eqref{eco1} is satisfied and thus there is nothing to prove.
Hence, we may assume that $n_2 \leq n-1$. Since $\alpha_2=\binom{n_2}{2}$, from \eqref{cond} it follows that 
\begin{equation}\label{eco4}
n_2(n_2-1) \geq (2n-q)(q-1).
\end{equation}
If $n_2\geq 2q-3$ and, henceforth $n\geq 2q-2$, then $f(n_2)\leq f(n)$ and, again, we are done from \eqref{eco3}. Thus, we can assume that $n_2\leq 2q-4$.
Also, if $n\geq 3q-4$, then from \eqref{ecos5} it follows that $f(n_2)\leq f(n)$ and there is nothing to prove. Thus $n\leq 3q-5$.
We consider three cases:
\begin{enumerate}
\item[(i)] $q=5$. From all the above, it follows that $n\leq 10$ and $n_2\leq 6$. Also, $n\geq q+2=7$. Therefore we have
$n_2(n_2-1)\leq 30$ and $(2n-q)(q-1)\geq (14-5)\cdot 4=36$. Thus \eqref{eco4} is impossible.
\item[(ii)] $q=6$. Similarly, we have $8\leq n\leq 13$ and $n_2\leq 8$. If $n_2\leq 7$, then $n_2(n_2-1)\leq 42$, while 
$(2n-q)(q-1)\geq 50$, thus \eqref{eco4} leads to a contradiction again. Also, if $n_2=8$, then $n\geq 9$ and thus
$$(2n-q)(q-1)\geq 60 > 56=n_2(n_2-1).$$ So, again, \eqref{eco4} does not hold.
\item[(iii)] $q=7$. We have $9\leq n\leq 16$ and $n_2\leq 10$. If $n_2\leq 8$, then $n_2(n_2-1)\leq 56$, while 
$(2n-q)(q-1)\geq 66$. If $n_2=9$, then $n\geq 10$ and $$(2n-q)(q-1)\geq 78\geq 72=n_2(n_2-1).$$
On the other hand, if $n_2=10$, then $n\geq 11$ and $(2n-q)(q-1)\geq 90=n_2(n_2-1)$. Hence, \eqref{eco4} is satisfied 
for $n=11$ and $n_2=10$. Since $\qdepth(S/I)=7$ it follows that
$$\beta_4^7(S/I)=\alpha_4-4\alpha_3+\binom{5}{2}\alpha_2-\binom{6}{3}n+\binom{7}{3}\geq 0,$$
and therefore 
\begin{equation}\label{eco7}
\alpha_4\geq 4\alpha_3 - 450+220-35=4\alpha_3 - 265.
\end{equation}
Since $\alpha_2=\binom{10}{2}$, from Lemma \ref{cord} it follows that $\alpha_4\leq 210$. Hence, from \eqref{eco7}
we get that $4\alpha_3 \leq 475$ and thus $\alpha_3\leq 118$. However, as $118=\binom{9}{3}+\binom{8}{2}+\binom{6}{1}$
this implies $\alpha_4\leq \binom{9}{4}+\binom{8}{3}+\binom{6}{2}=197$, which from \eqref{eco7} implies again that
$4\alpha_3\leq 265+197=462$. Thus $\alpha_3\leq 115$. It follows that
$$\alpha_3-(q-2)\alpha_2=115-5\cdot \binom{10}{2}=-110=\binom{11}{3}-5\cdot\binom{11}{2},$$
hence \eqref{eco2} is satisfied.
\end{enumerate}
Now, assume that $\alpha_2=\binom{n_2}{2}+\binom{n_1}{1}$, where $n_2>n_1\geq 1$ and $n_2\leq n-1$. From Lemma \ref{cord} it
follows that $\alpha_3\leq \binom{n_2}{3}+\binom{n_1}{2}$. Therefore, we have that:
\begin{equation}\label{eco8}
\alpha_3-(q-2)\alpha_2 \leq  \binom{n_2}{3}-(q-2)\binom{n_2}{2}+ \binom{n_1}{2} - (q-2) \binom{n_1}{1}.
\end{equation}
On the other hand, from \eqref{cond} it follows that
\begin{equation}\label{eco9}
(n_2+2)(n_2-1)\geq n_2(n_2-1)+2n_1 = 2\alpha_2 \geq (2n-q)(q-1).
\end{equation}
We consider the function
$$g(y)=\binom{y}{2}-(q-2)y,\text{ where }y\geq 0.$$
Since $g'(y)=y-(q-\frac{3}{2})$ it follows that
$$g(0)=0, g(y) \searrow \text{ on }[0,q-2],\;g(q-2)=g(q-1)=-\binom{q-1}{2},$$
\begin{equation}\label{eco55}
g(y) \nearrow \text{ on }[q-1,\infty)\text{ and }g(2q-3)=0.
\end{equation}
From \eqref{eco2} and \eqref{eco8}, in order to complete the proof we have to show that
\begin{equation}\label{desiree}
f(n_2)+g(n_1) \leq f(n) = \frac{n(n-1)(n-3q+4)}{6}.
\end{equation}
Assume that $n_2\geq 2q-3$. We have $n\geq 2q-2$. If $n_1\leq 2q-3$, then from \eqref{eco5} and \eqref{eco55}
it follows that $$f(n_2)+g(n_1)\leq f(n_2)\leq f(n),$$
and thus we are done. Now, assume that $n_1\geq 2q-2$. We have $n_2\geq 2q-1$ and $n\geq 2q$. 
From \eqref{eco5} and \eqref{eco55}, by straightforward computations, it follows that
\begin{align*}
& f(n)-(f(n_2)+g(n_1))\geq f(n_2+1)-f(n_2)-g(n_1) \geq \\
& \geq f(n_2+1)-f(n_2)-g(n_2-1) = n_2-q+1>0,
\end{align*}
and thus there is nothing to prove.

Now, we consider the case $n_2\leq 2q-4$. It follows that $1\leq n_1\leq 2q-5$ and $g(n_1)<0$. 
If $n\geq 3q-4$, from \eqref{ecos5} we get
$$f(n)\geq f(n_2)>f(n_2)+g(n_1),$$
and thus we are done. Hence, we may assume that $n\leq 3q-5$.
We consider three cases:
\begin{enumerate}
\item[(i)] $q=5$. We have $7\leq n\leq 10$. Also $n_2\leq 6$ and $n_1\leq 5$. 
From \eqref{eco9} it follows that $n=7$, $n_2=6$ and $n_1\in\{3,4,5\}$. Since $q-2=3$, we have $g(n_1)\leq g(5)=-5$ and therefore
$$f(n_2)+g(n_1) \leq f(6)+g(5) = -25-5 = -30 \leq -28 = f(7)=f(n),$$
and thus the required conclusion follows from \eqref{desiree}.
\item[(ii)] $q=6$. We have $8\leq n\leq 13$. Also $n_2\leq \min\{n-1,8\}$. If $n\geq 11$, then \eqref{eco9} leads to contradiction.
If $n=10$, then \eqref{eco9} implies $n_2=8$ and $n_1=7$. Hence
$$f(n_2)+g(n_1) = f(8)+g(7) = -63 < f(10)=-60,$$
and thus are done by \eqref{desiree}. 

If $n=9$, then \eqref{eco9} implies $n_2=8$ and $2\leq n_1\leq 7$. Since $g(n_1)\leq -7$ for all $2\leq n_1\leq 7$ it
follows that
$$f(n_2)+g(n_1)=f(8)+g(n_1)\leq - 56-7 = -64 < f(9) = -60,$$ 
and thus are done by \eqref{desiree}.

If $n=8$, then \eqref{eco9} implies $n_2=7$ and $4\leq n_1\leq 6$. Note that $g(4)=g(5)=-10$ and $g(6)=-9$.
It follows that
$$f(7)+g(n_1)\leq -49-9 = -58 < -56 = f(8),$$ 
and thus we are done by \eqref{desiree}.

\item[(iii)] $q=7$. We have $9\leq n\leq 16$. Also $n_2\leq \min\{n-1,10\}$. If $n\geq 13$, then \eqref{eco9} leads to contradiction.

If $n=12$, then \eqref{eco9} implies $n_2=10$ and $6\leq n_1\leq 9$. It follows that
$$f(n_2)+g(n_1) \leq f(10) + g(9) = -114 < -110 = f(12),$$
and we are done by \eqref{desiree}.

If $n=11$, then \eqref{eco9} implies $n_2=10$. Also, $1\leq n_1\leq 9$. It follows that
$$f(n_2)+g(n_1) \leq f(10) + g(1) = -110 = f(11),$$
and we are done by \eqref{desiree}.

If $n=10$, then \eqref{eco9} implies $n_2=9$ and $3\leq n_1\leq 8$. It follows that
$$f(n_2)+g(n_1) \leq f(9) + g(3) = -108 < -105 = f(10),$$
and we are done by \eqref{desiree}.

Finally, if $n=9$, then \eqref{eco9} implies $n_2=8$ and $5\leq n_1\leq 7$.
It follows that
$$f(n_2)+g(n_1) \leq f(8) + g(7) = -98 < -96 = f(9),$$
and we are done by \eqref{desiree}.
\end{enumerate}
\end{proof}

Note that, if $q\geq 8$, then the conclusion of Lemma \ref{b3q} does not hold in general; see Example \ref{minus2}.
In the following, we will assume that $q=\qdepth(S/I)=5$.

\begin{lema}\label{b45}
With the above notations, we have $\beta_4^5(S/I) \leq \binom{n-2}{4}$.
\end{lema}

\begin{proof}
We denote $\alpha_j:=\alpha_j(S/I)$ for all $j\geq 0$, and $\beta_k^q=\beta_k^q(S/I)$ for all $0\leq k\leq q$.
Since $\qdepth(S/I)\geq 5$ it follows that $n\geq 7$ and
\begin{equation}\label{cucuruz}
\beta_2^5= \alpha_2 - 4n + 10 \geq 0 \text{ and } \beta_3^5 = \alpha_3-3\alpha_2+6n-10 \geq 0.
\end{equation}
From \eqref{cucuruz} it follows that
\begin{equation}\label{papusoi}
\alpha_2\geq 4n-10,\; 3\alpha_2 \leq \alpha_3+6n-10 \text{ and }\alpha_3\geq 3\alpha_2-6n+10 \geq 6n-20.
\end{equation}
On the other hand, the conclusion is equivalent to
\begin{equation}\label{cucur}
\beta_4^5=\alpha_4-2\alpha_3+3\alpha_2-4n+5\leq \binom{n-2}{4}.
\end{equation}
We consider the functions
$$
f(x)=\binom{x}{4}-2\cdot \binom{x}{3},\; g(x)=\binom{x}{3}-2\cdot \binom{x}{2}\text{ and }
h(x)=\binom{x}{2}-2\cdot \binom{x}{1}.$$
Note that $f(x)=\frac{1}{24}x(x-1)(x-2)(x-11)$ and we have:

\begin{table}[htb]
\centering
\begin{tabular}{|l|l|l|l|l|l|l|l|l|l|l|l|l|l|l|}
\hline
$x$   & 1 & 2 & 3  & 4  & 5   & 6   & 7   & 8   & 9   & 10  & 11 & 12  \\ \hline
$f(x)$ & 0 & 0 & -2 & -7 & -15 & -25 & -35 & -42 & -42 & -30 & 0  & 55  \\ \hline
\end{tabular}
\end{table}
Also, $f$ is decreasing on $[2,8]$ and is increasing on $[9,\infty)$.

We have that $g(x)=\frac{1}{6}x(x-1)(x-8)$ and:
\begin{table}[htb]
\centering
\begin{tabular}{|l|l|l|l|l|l|l|l|l|l|l|l|l|l|l|}
\hline
$x$   & 1 & 2 & 3  & 4  & 5   & 6   & 7   & 8   & 9   \\ \hline
$g(x)$ & 0 & -2 & -5 & -8 & -10 & -10 & -7 & 0 & 12   \\ \hline
\end{tabular}
\end{table}

Also, $g$ is decreasing on $[1,5]$ and is increasing on $[6,\infty)$.

We have that $h(x)=\frac{1}{2}x(x-5)$ and:
\begin{table}[htb]
\centering
\begin{tabular}{|l|l|l|l|l|l|l|l|l|l|l|l|l|l|l|}
\hline
$x$    & 1 & 2 & 3  & 4  & 5   & 6 & 7 & 8 \\ \hline
$h(x)$ & -2 & -3 & -3 & -2 & 0 & 3 & 7 & 12 \\ \hline
\end{tabular}
\end{table}

Also, $h$ is increasing on $[3,\infty)$.

We consider several cases:
\begin{enumerate}
\item[(a)] $\alpha_3=\binom{n_3}{3}$. From Lemma \ref{cord} it follows that $\alpha_4\leq \binom{n_3}{4}$ and $\alpha_2\geq \binom{n_3}{2}$.
If $n_3=n$, then $\alpha_2=\binom{n}{2}$ and therefore
$$\beta_4^5 \leq \binom{n}{4}-2\binom{n}{3}+3\binom{n}{2}-4n+5 = \binom{n-2}{4},$$
as required. Hence, we can assume that $n_3<n$. 

Since $\alpha_2\leq \binom{n}{2}$ and $\alpha_4\leq \binom{n_3}{4}$,
a sufficient condition to have $\beta_4^5 \leq \binom{n-2}{4}$ is $f(n_3)\leq f(n)$.
This is clearly satisfied for $n\geq 11$; see the table with values of $f(x)$. 
Thus, we may assume that $n\leq 10$. We consider the subcases:
\begin{enumerate}
\item[(i)] $n=10$. Since $f(10)=-30$, the inequality $f(n_3)\leq f(10)$ is satisfied for $n_3\in\{7,8,9\}$ and there is nothing to prove.
           Thus $n_3\leq 6$ and $\alpha_3\leq \binom{6}{3}=20$. From \eqref{papusoi} it follows that
					 $$\alpha_2 \geq 4n-10 = 30\text{ and }3\alpha_2 \leq \alpha_3+6n-10 \leq 70,$$
					 which is impossible. 
\item[(ii)] $n=9$. Since $f(8)=f(9)$ we can assume that $n_3\leq 7$ and, therefore, 
            $\alpha_3\leq \binom{7}{3}=35$ and $\alpha_4\leq \binom{7}{4}=35$.
            From \eqref{papusoi} it
            follows that $$\alpha_2 \geq 4n-10=26\text{ and }3\alpha_2 \leq \alpha_3+6n-10\leq 79.$$
						Thus $\alpha_2=26$. Also, from \eqref{papusoi} it follows that $\alpha_3\geq 6n-20=34$ and thus $\alpha_3=35$.
						We get						
						$$\beta_4^5 
						\leq 35 - 2\cdot 35 + 3\cdot 26 - 36 + 5 = 12 \leq \binom{9-2}{4}=35,$$
						as required.
\item[(iii)] $n=8$. We have $n_3\leq 7$ and thus $\alpha_3\leq \binom{7}{3}=35$ and $\alpha_4\leq \binom{7}{4}=35$. 
             From \eqref{papusoi} it follows
             that $$3\alpha_2 \leq \alpha_3 + 38 \leq 73\text{ and }\alpha_3\geq 28>\binom{6}{3}.$$
						 Therefore $\alpha_3=35$ and $\alpha_2\leq 24$. We get:
						 $$\beta_4^5 
						 \leq 35 - 2\cdot 35 + 3\cdot 24 - 32 + 5 = 10 \leq \binom{8-2}{4}=15,$$
						 as required.
\item[(iv)]  $n=7$. We have $n_3\leq 6$ and thus $\alpha_3\leq \binom{6}{3}=20$. From \eqref{papusoi} it follows
             that $$\alpha_2\geq 4\cdot 7-10=18\text{ and } 3\alpha_2 \leq \alpha_3 + 6\cdot 7 - 10 \leq 52,$$
						 which is impossible.
\end{enumerate} 
\item[(b)] $\alpha_3=\binom{n_3}{3}+\binom{n_2}{2}$ with $n>n_3>n_2\geq 2$.
           From Lemma \ref{cord} it follows that $\alpha_4\leq \binom{n_3}{4} + \binom{n_2}{3}$ and 
					 $\alpha_2\geq \binom{n_3}{2}+\binom{n_2}{1}$. Similar to the case (a), if $f(n_3)+g(n_2)\leq f(n)$,
					 then $\beta_4^5 \leq \binom{n-2}{4}$ and there is nothing to prove. 									
					 If $n\geq 11$ and $n_2\leq 8$, then, from the tables of values of $f(x)$ and $g(x)$ we have
					 $$f(n_3)+g(n_2)\leq f(n_3) \leq f(n),$$
					 and there is nothing to prove. If $n\geq 11$ and $n_2\geq 9$, then \small
					 $$f(n)-(f(n_3)+g(n_2))\geq f(n)-(f(n-1)+g(n-2)) =\frac{1}{2}(n^2-9n+14) >0, $$ \normalsize
					 and we are done. Hence, we may assume that $n\leq 10$. 
	         We consider the subcases:
\begin{enumerate}
\item[(i)] $n=10$. Similar to the case (a.i), if $n_3\in\{7,8,9\}$, then $f(n_3)+g(n_2)\leq f(10)$ and we are done.
           If $n_3\leq 6$, then $\alpha_3\leq \binom{6}{3}+\binom{5}{2}=30$. On the other hand, from \eqref{papusoi}
					 it follows that $\alpha_3\geq 6\cdot 10 - 20 = 40$, which gives a contradiction.
\item[(ii)] $n=9$. If $n_3=8$, then $f(n_3)+g(n_2)\leq f(n_3)=f(8)=f(9)$ and we are done. If $n_3=7$ and $n_2\in \{4,5,6\}$,
            then $f(n_3)+g(n_2)\leq f(7)+g(4)=-43<-42=f(9)$ and we are also done. If $n_3=7$ and $n_2\leq 3$, then
						$\alpha_3\leq \binom{7}{3}+\binom{3}{2}=38$, $\alpha_3\geq \binom{7}{3}+\binom{2}{2}=36$ 
						and $\alpha_4\leq \binom{7}{4}+\binom{3}{3}=36$. From \eqref{papusoi} it
            follows that $$3\alpha_2 \leq \alpha_3+6n-10\leq 82,\text{ hence }\alpha_2\leq 27.$$
						Therefore $$\beta_4^5 \leq 36 - 2\cdot 36 + 3 \cdot 27 - 36 + 5 = 14 \leq \binom{n-2}{4}=\binom{7}{4}=35,$$
						as required. Now, if $n_3\leq 6$, then, as in the case (i), we have $\alpha_3\leq 30$, which contradict the fact
						that $\alpha_3\geq 6\cdot 9 - 20 = 34$.
\item[(iii)] $n=8$. As in the case (ii), if $n_3=7$ and $n_2\in\{4,5,6\}$, then there is nothing to prove. Also,
             if $n_3=7$ and $n_2\leq 3$, then $\alpha_3\in \{36,38\}$ and $\alpha_4\leq 36$. From \eqref{papusoi} it follows that
						 $3\alpha_2\leq \alpha_3+6n-10\leq 76,\text{ hence }\alpha_2\leq 25.$
             Therefore
						 $$\beta_4^5 \leq 36 - 2\cdot 36 + 3\cdot 25 - 32 + 5 = 12 \leq \binom{8-2}{4}=15,$$
						 and we are done. Now, if $n_3\leq 6$, then $\alpha_3\leq \binom{6}{3}+\binom{5}{2}=30$ and thus 
						 $\alpha_4\leq \binom{6}{4}+\binom{5}{3}=25$. Also, from \eqref{papusoi} we have 
						 $$\alpha_3\geq 6\cdot 8 - 20=28\text{ and }3\alpha_2 \leq 30+38 = 68.$$
						 Hence $\alpha_2\leq 22$. It follows that
						 $$\beta_4^5 \leq 25 - 2 \cdot 28 + 3\cdot 22 - 32 + 5 = 8 \leq \binom{6}{4}=15,$$
						 and we are done.						
\item[(iv)]  $n=7$. If $n_3=6$ and $n_2=5$, then $f(n_3)+g(n_2)=-25-10=-35=f(n)$ and there is nothing to prove.
             Assume that $n_3\leq 6$ and $n_2\leq 4$. It follows that $\alpha_3\leq \binom{6}{3}+\binom{4}{2}=26$
						 and $\alpha_4\leq \binom{6}{4}+\binom{4}{3}=19$. From \eqref{papusoi} we get
						 $$\alpha_2\geq 4\cdot 7-10=18\text{ and } 3\alpha_2 \leq \alpha_3 + 6\cdot 7 - 10 \leq 58.$$
						 Hence $\alpha_2\in\{18,19\}$. Also, $\alpha_3\geq 6\cdot 7 - 20 = 22$ and thus 
						 $\alpha_3\geq \binom{6}{3}+\binom{3}{2}=23$. If $\alpha_3=26$, then $\alpha_4\leq 19$ and
						 $$\beta_4^5 \leq 19 - 2\cdot 26 + 3\cdot 19 - 28 + 5 = 1\leq \binom{7-2}{4}=5,$$
						 and we are done. If $\alpha_3=23$, then $\alpha_4\leq \binom{6}{4}+\binom{3}{3}=16$ and
						 $$\beta_4^5\leq 16 - 2\cdot 23 + 3\cdot 19 - 28 + 5 = 4\leq \binom{7-2}{4}=5,$$
						 and we are done, also.
\end{enumerate} 				
\item[(c)] $\alpha_3=\binom{n_3}{3}+\binom{n_2}{2}+\binom{n_1}{1}$ with $n>n_3>n_2>n_1\geq 1$.
           From Lemma \ref{cord} it follows that $\alpha_4\leq \binom{n_3}{4} + \binom{n_2}{3}+\binom{n_1}{2}$ and 
					 $\alpha_2\geq \binom{n_3}{2}+\binom{n_2}{1}$. Similar to the previous cases, if $f(n_3)+g(n_2)+h(n_1)\leq f(n)$,
					 then $\beta_4^5 \leq \binom{n-2}{4}$ and there is nothing to prove. If $n\geq 11$, then \small
					 $$f(n)-f(n_3)-g(n_2)-h(n_1)\geq f(n)-f(n-1)-g(n-2)-h(n-3) = n-5 \geq 6,$$ \normalsize
					 and we are done. Hence, we may assume that $n\leq 10$. We consider the subcases:
\begin{enumerate}
\item[(i)] $n=10$. If $n_3\in\{7,8,9\}$, then, looking at the values of $f(x),g(x)$ and $h(x)$, it is easy to check that
           $f(n_3)+g(n_2)+h(n_1)\leq f(10)$, and we are done. If $n_3\leq 6$, then $\alpha_3\leq \binom{6}{3}+\binom{5}{2}+\binom{4}{1}=34$. 
					 On the other hand, from \eqref{papusoi} it follows that $\alpha_3\geq 6\cdot 10 - 20 = 40$, which gives a contradiction.
\item[(ii)] $n=9$. If $n_3=8$, then $n_2\leq 7$ and $n_1\leq 6$. In particular, $g(n_2)\leq 0$. If $n_1\leq 5$, then $h(n_1)\leq 0$.
            If $n_1=6$, then $n_2=7$ and $g(n_2)+h(n_1)=-7+3=-4\leq 0$. It follows that
            $$f(n_3)+g(n_2)+h(n_1)\leq f(n_3)=f(8)=-42=f(9)=f(n),$$
						and we are done. If $n_3=7$ and $n_2\in \{4,5,6\}$, then $n_1\leq 5$, $h(n_1)\leq 0$ and, moreover, we have
						$$f(n_3)+g(n_2)+h(n_1)\leq f(n_3)+g(n_2)\leq -35-8=-43 < -42=f(n),$$
						are we are also done. If $n_3=7$ and $n_2\leq 3$, then $n_1\leq 2$ and therefore
						$$\binom{7}{3}+\binom{2}{2}+\binom{1}{1} =37 \leq \alpha_3\leq 40 = \binom{7}{3}+\binom{3}{2}+\binom{2}{1}.$$
						From \eqref{papusoi} it follows that 
						$$3\alpha_2 \leq \alpha_3+6n-10\leq 84,\text{ hence }\alpha_2\leq 28.$$
						Note that $\alpha_4\leq \binom{7}{4}+\binom{3}{3}+\binom{2}{2}=37$.
						Therefore $$\beta_4^5 \leq 37 - 2\cdot 37 + 3 \cdot 28 - 36 + 5 = 16 \leq \binom{n-2}{4}=\binom{7}{4}=35,$$
						as required. On the other hand, if $n_3\leq 6$, then $n_2\leq 5$ and $n_1\leq 4$ and thus
						$\alpha_3\leq \binom{6}{3}+\binom{5}{2}+\binom{4}{1}=34$. Also, from \eqref{papusoi} it follows that 
						$\alpha_3\geq 6n-20=34$. Thus $\alpha_3=34$ and $\alpha_4\leq \binom{6}{4}+\binom{5}{3}+\binom{4}{2}=31$.
						Also, from \eqref{papusoi} it follows that 
						$$3\alpha_2 \leq \alpha_3+6n-10\leq 78,\text{ hence }\alpha_2\leq 26.$$											
						Therefore $$\beta_4^5 \leq 31 - 2\cdot 34 + 3\cdot 26 - 36 + 5 = 10 \leq \binom{7}{4}.$$
						Hence, we are done.
\item[(iii)] $n=8$. As in the case (ii), if $n_3=7$ and $n_2\in\{4,5,6\}$, then there is nothing to prove. Also, using the argument from the case (ii), 
             if $n_2\leq 3$, then $37\leq \alpha_3\leq 40$ and $\alpha_4\leq 37$. From \eqref{papusoi} it follows that
						 $$3\alpha_2 \leq \alpha_3+6n-10 \leq 78,\text{ hence }\alpha_2\leq 26.$$						
						 Therefore $$\beta_4^5 \leq 37-2\cdot 37+3\cdot 26 - 32 + 5 = 14\leq 15=\binom{6}{4}=\binom{n-2}{4},$$
						 and we are done. Now, if $n_3\leq 6$, then $\alpha_3\leq 34$ and thus $\alpha_4\leq 31$.
						 Also, from \eqref{papusoi} it follows that $\alpha_3\geq 6n-20=28$ and 
						 $$3\alpha_2 \leq \alpha_3+6n-10\leq 72,\text{ hence }\alpha_2\leq 24.$$
						 If $\alpha_2\leq 22$, then 
						 $$\beta_4^5 \leq 31 - 2\cdot 28 + 3\cdot 22 - 32 + 5 = 14 \leq \binom{8-2}{4}=15,$$
						 and we are done. Similarly, if $\alpha_3\geq 31$, then 
						 $$\beta_4^5 \leq 31 - 2\cdot 31 + 3\cdot 24 - 32 + 5 = 14 \leq \binom{8-2}{4}=15,$$
						 and we are done. So we may assume $\alpha_3\in\{28,29,30\}$ and $\alpha_2\in\{23,24\}$.
						 
						 If $\alpha_3=28=\binom{6}{3}+\binom{4}{2}+\binom{2}{1}$, then, from Lemma \ref{cord}, it 
						 follows that $\alpha_4\leq \binom{6}{4}+\binom{4}{3}+\binom{2}{2}=20$ and thus
						 $$\beta_4^5 \leq 20 - 2\cdot 28 + 3\cdot 24 - 32 + 5 = 9 < 15,$$
						 as required. If $\alpha_3=29=\binom{6}{3}+\binom{4}{2}+\binom{3}{1}$, then, from Lemma \ref{cord}, it 
						 follows that $\alpha_4\leq \binom{6}{4}+\binom{4}{3}+\binom{3}{2}=22$ and thus
						 $$\beta_4^5 \leq 22 - 2\cdot 29 + 3\cdot 24 - 32 + 5 = 9 < 15,$$
						 as required.  If $\alpha_3=29=\binom{6}{3}+\binom{5}{2}$, then, from Lemma \ref{cord}, it 
						 follows that $\alpha_4\leq \binom{6}{4}+\binom{5}{3}=25$ and thus
						 $$\beta_4^5 \leq 25 - 2\cdot 30 + 3\cdot 24 - 32 + 5 = 10 < 15,$$
						 and we are done.
\item[(iv)] $n=7$. If $n_3=6$ and $n_2=5$, then $n_1\leq 4$ and therefore 
            $$f(n_3)+g(n_2)+h(n_1)\leq f(n_3)+g(n_2)=-25-10=-35=f(n),$$ and we are done.						
						Assume that $n_3\leq 6$ and $n_2\leq 4$. It follows that $\alpha_3\leq \binom{6}{3}+\binom{4}{2}+\binom{3}{1}=29$
						 and $\alpha_4\leq \binom{6}{4}+\binom{4}{3}+\binom{3}{2}=22$. From \eqref{papusoi} we get
						 $$\alpha_2\geq 4\cdot 7-10=18\text{ and } 3\alpha_2 \leq \alpha_3 + 6\cdot 7 - 10 \leq 61.$$
						 Hence $\alpha_2\in\{18,19,20\}$. Also, $\alpha_3\geq 6\cdot 7 - 20 = 22$. If $\alpha_2=18=\binom{6}{2}+\binom{3}{1}$,
						 then, from Lemma \ref{cord}, it follows that $\alpha_4\leq \binom{6}{4}+\binom{3}{3}=16$ and thus
						 $$\beta_4^5\leq 16 - 2\cdot 22 + 3 \cdot 18 - 28 + 5 = 3 < 5 =\binom{7-2}{4}.$$
						 If $\alpha_2\geq 19$, then $\alpha_3\geq 25$ and thus 
						 and thus 
						 $\alpha_3\in\{25,28,29\}$.					
						 If $\alpha_3=25$, then $\alpha_4\leq 17$ and thus $\alpha_4-2\alpha_3\leq -33$. Also
						 $3\alpha_2\leq 25+32 = 57$ and thus $\alpha_2\leq 19$. It follows that
						 $$\beta_4^5\leq - 33 + 3\cdot 19 - 28 + 5 = 1 \leq \binom{7-2}{4}=5,$$
						 and we are done. If $\alpha_3\geq 28$, then 
						 $$\beta_4^5\leq 22 - 2\cdot 28 + 3\cdot 20 - 28 + 5 = 3 \leq \binom{7-2}{4}=5,$$
						 and we are also done.						 
\end{enumerate}
\end{enumerate}
Hence, the proof is complete.
\end{proof}

\begin{lema}\label{b55}
With the above notations, we have that $\beta_5^5\leq \binom{n-1}{5}$.
\end{lema}

\begin{proof}
Since $\beta_5^5=\alpha_5-\beta_4^4$ and $\beta_4^4\geq 0$, the conclusion follows immediately if $\alpha_5\leq \binom{n-1}{5}$.
Hence, we may assume that $$\alpha_5\geq \binom{n-1}{5}+1 = \binom{n-1}{5}+\binom{4}{4}.$$
From Lemma \ref{cord} it follows that $\alpha_j\geq \binom{n-1}{j}+\binom{4}{j}$ for all $j\in\{2,3,4\}$.
We claim that 
\begin{equation}\label{claie}
\alpha_3-\alpha_2 \leq \binom{n}{3} - \binom{n}{2} = \frac{1}{6}n(n-1)(n-5).
\end{equation}
If $\alpha_2=\binom{n}{2}$, then there is nothing to prove. 
Assume that $\alpha_2=\binom{n-1}{2}+\binom{n_1}{1}$, where $n-1>n_1\geq 4$. Then, from Lemma \ref{cord}, we have
$\alpha_3\leq \binom{n-1}{3}+\binom{n_1}{2}$. It follows that
\begin{align*}
& \alpha_3-\alpha_2\leq \binom{n-1}{3}+\binom{n_1}{2}-\binom{n-1}{2}-\binom{n_1}{1} =\\
& = \frac{1}{6}(n-1)(n-2)(n-6) + \frac{1}{2}n_1(n_1-3)\leq \\
& \leq \frac{1}{6}(n-1)(n-2)(n-6) + \frac{1}{2}(n-2)(n-5) = \\
& = \frac{1}{6}(n^3-6n^2-n+18) < \frac{1}{6}n(n-1)(n-5),
\end{align*}
and thus the claim \eqref{claie} holds. Hence, in order to complete the proof it is enough to show that
\begin{equation}\label{wish}
\alpha_5-\alpha_4\leq \binom{n}{5}-\binom{n}{4}=\frac{1}{120}n(n-1)(n-2)(n-3)(n-9).
\end{equation}
We denote $f_k(x)=\binom{x}{k}-\binom{x}{k-1}$ for all $2\leq k\leq 5$. We have the following table of values:

\begin{table}[tbh]
\centering
\label{Tab:Tcr}
\begin{tabular}{|l|l|l|l|l|l|l|l|l|l|}
\hline
$x$      & 1  & 2  & 3  & 4  & 5  & 6  & 7   & 8   & 9  \\ \hline
$f_5(x)$ & 0  & 0  & 0  & -1 & -4 & -9 & -14 & -14 & 0  \\ \hline
$f_4(x)$ & 0  & 0  & -1 & -3 & -5 & -5 & 0   & 14  & 42 \\ \hline
$f_3(x)$ & 0  & -1 & -2 & -2 & 0  & 5  & 14  & 28  & 48 \\ \hline
$f_2(x)$ & -1 & -1 & 0  & 2  & 5  & 9  & 14  & 20  & 27 \\ \hline
\end{tabular}
\end{table}

We consider several cases:
\begin{enumerate}
\item[(a)] $\alpha_4=\binom{n-1}{4}+\binom{n_3}{3}$, where $n-1>n_3\geq 4$. If $n=7$, then $n_3\in\{4,5\}$ and therefore 
           $\binom{n_3}{4}-\binom{n_3}{3}\in \{-3,-5\}$. Also, we have $\alpha_5\leq \binom{6}{5}+\binom{n_3}{4}$. If $n_3=5$, then
					 $$\alpha_5-\alpha_4 \leq \binom{6}{5}-\binom{6}{4}-5 = -14 = \binom{7}{5}-\binom{7}{4},$$
					 and the condition \eqref{wish} is fulfilled. Now, assume that $n_3=4$, that is $\alpha_4=\binom{6}{4}+\binom{4}{3}=19$. Since
					 $$\beta_4^4=\alpha_4-(\alpha_3-\alpha_2)-7+1 = 13-(\alpha_3-\alpha_2)\geq 0,$$
					 it follows that $\alpha_3-\alpha_2\leq 13$. Since $\alpha_5 \leq \binom{6}{5}+\binom{4}{4}=7$,								
					 we have $\alpha_5-\alpha_4\leq -12$. If $\alpha_3-\alpha_2\leq 12$, then												
					 $$\beta_5^5=(\alpha_5-\alpha_4)+(\alpha_3-\alpha_2)+7-1\leq 6 =\binom{7-1}{5},$$
					 as required. So, we may assume $\alpha_3-\alpha_2=13$. Note that, if $\alpha_2 =\binom{a}{2}+\binom{b}{1} \leq 20$
					 with $6\geq a >b\geq 1$, then, from Lemma \ref{cord}, we get 
					 $$\alpha_3-\alpha_2 \leq f_3(a)+f_2(b) \leq f_3(6)+f_2(5)=10,$$ a contradiction.
					 Hence $\alpha_2=21$ and $\alpha_3=34$. We get
					 $$\beta_4^5 = \alpha_4 - 2\alpha_3 + 3\alpha_2 - 4\cdot 7 + 5 = 19-2\cdot 34 +3\cdot 21 - 28 +5 = -9,$$
					 a contradiction with the fact that $\qdepth(S/I)=5$.
					
					 If $n=8$, then $n_3\in\{4,5,6\}$. It follows that
					 $$\alpha_5-\alpha_4=f_5(7)+f_4(n_3)\leq -14-3=-17 < -14=f_5(8),$$
					 and we are done. If $n\geq 9$, then
					 \begin{align*}
					 & f_5(n)-(\alpha_5-\alpha_4)=f(n)-(f_5(n-1)+f_4(n_3))\geq \\
					 & \geq f_5(n)-(f_5(n-1)+f_4(n-2)) = \frac{1}{6}(n-2)(n-3)(n-7)>0,
					 \end{align*}
					 and thus we are done.					
\item[(b)]	$\alpha_4=\binom{n-1}{4}+\binom{n_3}{3}+\binom{n_2}{2}$, where $n-1>n_3\geq 4$ and $n_3>n_2\geq 2$.
            Assume $n=7$. Since $\alpha_5-\alpha_4 \leq f_5(6) + f_4(n_3) + f_3(n_2)$, the only case in which
						$\alpha_5-\alpha_4$ could be larger than $f_5(7)=-14$ is $n_3=4$ and $n_2=2$. This means 
						$$\alpha_4=\binom{6}{4}+\binom{4}{3}+\binom{2}{2}=15+4+1=20.$$
						Since $\beta_3^5\geq 0$ and $\beta_4^5\geq 0$ it follows that
						$$0 \leq \beta_3^5 + \beta_4^5 = \alpha_4-\alpha_3+2n-5=29-\alpha_3,$$
						thus $\alpha_3\leq 29$. On the other hand, we have
						$$\alpha_5\leq \binom{6}{5}+\binom{4}{4}=7\text{ and }\alpha_2\geq \binom{6}{2}+\binom{4}{1}+\binom{2}{0}=20.$$
						It follows that 					  
						$$\beta_5^5 = \alpha_5-\alpha_4+\alpha_3-\alpha_2+6 \leq 7-20+29-20+6 = 2 \leq \binom{7-1}{5}=6,$$
						and we are done. If $n=8$, then $4\leq n_3\leq 6$ and $2\leq n_2\leq 5$. In particular, $f_4(n_3)\leq -3$ and $f_3(n_2)\leq 0$.
						Thus $$\alpha_5-\alpha_4 = f_5(7)+f_4(n_3)+f_3(n_2)\leq f_5(7)-3=-17<f_5(8)=-14,$$
						and we are done. Now, assume $n\geq 9$. Then
						\begin{align*}
						& f_5(n)-(\alpha_5-\alpha_4) = f_5(n)-(f_5(n-1)+f_4(n_3)+f_3(n_2))\geq \\		
						& \geq f_5(n)-(f_5(n-1)+f_4(n-2)+f_3(n-3)) = \frac{(n-3)(n-6)}{2} > 0.
						\end{align*}
						Hence, we get the required result.
\item[(c)]	$\alpha_4=\binom{n-1}{4}+\binom{n_3}{3}+\binom{n_2}{2}+\binom{n_1}{1}$, where $n-1>n_3\geq 4$ and $n_3>n_2>n_1\geq 1$.
            Similar to the previous cases, $\alpha_5\leq \binom{n-1}{5}+\binom{n_3}{4}+\binom{n_2}{3}+\binom{n_1}{2}$ and thus
						$$\alpha_5-\alpha_4\leq f_5(n-1)+f_4(n_3)+f_3(n_2)+f_2(n_1).$$
						If $n=7$, then $n_3\geq 4$, $n_2\geq 2$ and $n_1\geq 1$. From the table with the values of $f_k(x)$'s, it is easy to
						check that $$f_5(6)+f_4(n_3)+f_3(n_2)+f_2(n_1) \leq -9-3-1-1 = -14 = f_5(7).$$
						If $n\geq 8$, then 
						\begin{align*}
						& f_5(n)-(\alpha_5-\alpha_4) = f_5(n)-(f_5(n-1)+f_4(n_3)+f_3(n_2)+f_2(n_1))\geq \\
						& \geq f_5(n)-(f_5(n-1)+f_4(n-2)+f_3(n-3)+f_2(n-4)) = n-5 > 0,
						\end{align*}
						and we are done, again.
\end{enumerate}
Hence, the proof is complete.
\end{proof}

\begin{teor}\label{main}
Let $I\subset S$ be a squarefree monomial ideal with $\hdepth(S/I)=5$. 
Then $\hdepth(I)\geq 5$.
\end{teor}

\begin{proof}
The conclusion follows from Lemma \ref{lem22}, Proposition \ref{cook}, Lemma \ref{b3q} (the case $q=5$), Lemma \ref{b45} and Lemma \ref{b55}.
\end{proof}

\begin{cor}\label{cmain}
Let $I\subset S=K[x_1,\ldots,x_n]$ be a squarefree monomial ideal, where $n\leq 7$. Then 
$\hdepth(I)\geq \hdepth(S/I)$.
\end{cor}

\begin{proof}
If $I$ is principal, which, by Theorem \ref{teo1}, is equivalent to $\qdepth(S/I)=n-1$, we
have, again, by Theorem \ref{teo1} that $\qdepth(I)=n$ and there is nothing to prove. Hence, we may assume that $q=\qdepth(S/I)\leq 5$.
If $q\leq 4$ we are done by Theorem \ref{teo2} and Theorem \ref{t3}. For $q=5$, the conclusion follows from Theorem \ref{main}.
\end{proof}

\end{document}